\documentclass[a4paper]{amsart}

\usepackage{amsmath}
\usepackage{amsthm}
\usepackage{amssymb}
\usepackage{amsfonts}
\usepackage{mathtools}
\usepackage{xfrac}
\usepackage{float}
\usepackage{fullpage}

\usepackage{newpxtext}
\usepackage{newpxmath}
\usepackage{hyperref}

\theoremstyle{plain}
\newtheorem{theorem}{Theorem}[section]
\newtheorem{lemma}[theorem]{Lemma}
\newtheorem{prop}[theorem]{Proposition}
\newtheorem{cor}[theorem]{Corollary}
\newtheorem*{theorem*}{Theorem}
\newtheorem*{observation*}{Observation}
\theoremstyle{definition}

\newtheorem{rem}[theorem]{Remark}
\newtheorem{observation}[theorem]{Observation}
\usepackage[dvipsnames]{xcolor}
\usepackage{tikz}
\usepackage{tikz-cd}
\usetikzlibrary{shapes}
\usetikzlibrary{decorations.pathreplacing}
\colorlet{lgray}{gray!25}
\colorlet{llgray}{gray!50}
\DeclareMathOperator{\Hom}{Hom}
\DeclareMathOperator{\res}{res}
\DeclareMathOperator{\tr}{tr}
\DeclareMathOperator{\coker}{coker}
\DeclareMathOperator{\im}{im}

\DeclareMathOperator{\colim}{colim}

\newcommand{\M}{\ensuremath{\underline{M}}}
\newcommand{\trm}{\ensuremath{\tr_{\M}}}
\newcommand{\resm}{\ensuremath{\res_{\M}}}
\newcommand{\A}{\ensuremath{\underline{\mathbb{A}}}}
\newcommand{\zt}{\ensuremath{\underline{\tilde{\mathbb{Z}}}}}
\newcommand{\ftwo}{\ensuremath{\underline{\mathbb{F}}_2}}
\newcommand{\aq}{\ensuremath{\mathbb{A}(Q)}}
\newcommand{\hm}{\ensuremath{H\underline{M}}}
\newcommand{\ha}{\ensuremath{H\underline{\mathbb{A}}}}
\newcommand{\hz}{\ensuremath{H\underline{\mathbb{Z}}}}
\newcommand{\hzt}{\ensuremath{H\underline{\tilde{\mathbb{Z}}}}}
\newcommand{\hftwo}{\ensuremath{H\underline{\mathbb{F}}_2}}
\newcommand{\hb}{\ensuremath{H\underline{B}}}
\newcommand{\ho}[2]{\ensuremath{\left(#1_{hQ}\right)_{#2}}}
\newcommand{\mackey}[4]{\ensuremath{\begin{tikzcd}
#1\dar[bend right,"#3"'] \\
#2\uar[bend right,"#4"']
\end{tikzcd}}}

\title{On the $RO(Q)$-graded coefficients of Eilenberg-MacLane spectra}
\author{Igor Sikora}
\begin{document}
\begin{abstract}
Let $Q$ denote the cyclic group of order two. Using the Tate diagram we compute the $RO(Q)$-graded coefficients of Eilenberg-MacLane $Q$-spectra and describe their structure as modules over the coefficients of the Eilenberg-MacLane spectrum of the Burnside Mackey functor. If the underlying Mackey functor is a Green functor, we also infer the multiplicative structure on the $RO(Q)$-graded coefficients.
\end{abstract}
\maketitle
\tableofcontents
\section{Introduction}
Let $G$ be a finite group. In $G$-equivariant topology the role of ordinary cohomology is played by \emph{Bredon cohomology} \cite{MR0214062}. Whilst easy to define, making computations in Bredon theories is more complicated than in their non-equivariant analogues. Firstly, the coefficients of Bredon theories are of the form of a functor from the orbit category of $G$ to the category of abelian groups. Secondly, Bredon theories are more naturally \emph{$RO(G)$-graded} (graded over the representation ring of $G$) rather than graded over the integers. 

As shown in \cite{MR598689}, a $\mathbb{Z}$-graded Bredon theory extends to an $RO(G)$-graded one if its coefficients are of the form of a \emph{Mackey functor}. As in non-equivariant topology, the Bredon homology/cohomology with coefficients in a Mackey functor $\underline{M}$ is represented by the \emph{Eilenberg-MacLane $G$-spectrum} $\hm$. Spectra of this form appear in various contexts in equivariant topology. For example, equivariant Eilenberg-MacLane spectra are $0$-slices in the slice spectral sequence \cite[Proposition 4.50]{MR3505179}. 

The difficulties of computations in $RO(G)$-graded Bredon theories may be seen in calculations of $\hm^G_\star:=\pi^G_\star (\hm)$, the $RO(G)$-graded $G$-homotopy groups of $\hm$ (a five-pointed star indicates $RO(G)$-grading). This is equivalent to computing the $RO(G)$-graded Bredon homology and cohomology of a point with coefficients in $\underline{M}$ and thus we will refer to $\hm^G_\star$ as the \emph{coefficients of $\hm$}. The groups $\hm^G_n$ are zero for $n\in\mathbb{Z}$ unless $n=0$, which resembles the non-equivariant case. However, if $V$ is not a trivial representation then $\hm^G_V$ might be non-zero.

In this paper we use the Tate diagram to compute the $RO(Q)$-graded coefficients of Eilenberg-MacLane $Q$-spectra, where $Q$ is the cyclic group of order $2$. We do this in three instances: as an $RO(Q)$-graded abelian group, as a module over the coefficients of the Eilenberg-MacLane spectrum associated to the Burnside Mackey functor, and finally as an $RO(Q)$-graded ring, when appropriate. 

\subsection*{Tate diagram}
The idea behind the Tate diagram is to decompose a $Q$-spectrum $X$ into computationally simpler pieces:
\begin{enumerate}
\item Borel completion $X^h$;
\item free $Q$-spectrum $X_h$;
\item singular spectrum $X^\Phi$;
\item Tate spectrum $X^t$,
\end{enumerate} 
which are connected by the following commutative diagram:
\[
\begin{tikzcd}
X_h \rar\dar["\simeq"] & X \rar\dar & X^\Phi\dar \\
X_h \rar & X^h \rar & X^t.
\end{tikzcd}
\]

What makes computations by the Tate diagram feasible is that the rows are cofibre sequences and the right-hand square (known as the \emph{Tate square}) is a homotopy pullback. Moreover, the coefficients of the spectra appearing in the bottom row may be computed by the \emph{homotopy orbit} and \emph{homotopy fixed point} and \emph{Tate spectral sequences}, respectively. The foundational work on the Tate diagram is \cite{MR1230773}, where all of the details are discussed.

The computational strength of the Tate diagram has been proven in various contexts: for example Greenlees used it in \cite{FourApproaches} to compute the coefficients of the Eilenberg-MacLane $Q$-spectrum $\hz$ as an $RO(Q)$-graded ring, Greenlees and Meier computed the coefficients of K-theory with reality $K\mathbb{R}$ in \cite[Section 11]{MR3709655} and Hu and Kriz used it to compute the coefficients of $\hftwo$ and the $Q$-equivariant Steenrod algebra in \cite[Section 6]{MR1808224}. It was also used to compute the coefficients of $\hz$ over groups $C_{p^2}$ with $p$ prime by Zeng in \cite{zeng2018equivariant}.

\subsection*{$RO(Q)$-graded abelian group structure} The first step is describing the $RO(Q)$-graded abelian group structure of $\hm^Q_\star$. We show that it is fully determined by the underlying Mackey functor $\M$. This structure is given in Theorem \ref{Theorem RO(Q)graded abelian structure of hmq}, which can be informally stated as follows:
\begin{theorem*}
The $RO(Q)$-graded abelian group structure of $\hm^Q_\star$ may be presented by Figure \ref{Figure introduction}, where:
\begin{enumerate}
\item Every lattice point represents a $Q$-representation, the horizontal axis describes multiplicity of the trivial $Q$-representation and the vertical axis describes multiplicity of the sign representation.
\item The empty circle at the position $(0,0)$ is the abelian group $\M(Q/Q)$.
\item The values of $\hm^Q_\star$ lying on the $x=0$ axis are subgroups of $\M(Q/Q)$ given by the kernel of the restriction and the cokernel of the transfer.
\item The full dots in positions $(1,-1)$ and $(-1,1)$ are $\mathbb{Z}[Q]$-submodules of $\M(Q/e)$ given by the kernel of the transfer and the cokernel of the restriction respectively, whereas the values lying on the red/blue lines above/below them are their subquotients.
\item The values lying in blue and red areas are respectively the group cohomology and homology of $Q$ with coefficients in $\M(Q/e)$.
\item All other values are zero. 
\end{enumerate}
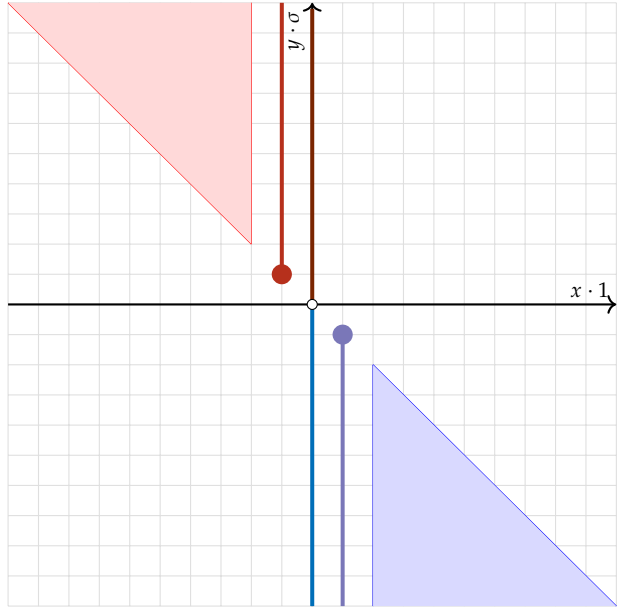
\begin{figure}[ht]
\begin{tikzpicture}[scale=0.4]
%Kratka

\draw[help lines, thin, lgray] (0,0) grid (20,20);
%\draw[llgray, dashed] (0,20)--(20,0);
\draw[->, thick] (0,10)--(20,10) ;
\draw[->, thick] (10,0)--(10,20) ;
%Homologie
\draw[red] (0,20)--(8,12);
\draw[red] (8,12)--(8,20);
\fill[red!15!white] (0,20)--(8,12)--(8,20);

%\node[scale=0.8] at (5.7,17.5) {$\ho{\hm}{x+y\sigma}$};

%Kohomologie
\draw[blue] (12,8)--(20,0);
\draw[blue] (12,8)--(12,0);
\fill[blue!15!white] (20,0)--(12,8)--(12,0);
%\node[scale=0.8] at (15,2) {$\hm^{hQ}_{x+y\sigma}$};

\draw[BrickRed, line width=1.5pt] (9,11)--(9,20);
\node[circle,BrickRed,fill,scale=0.8] at (9,11) {};
%\node[scale=0.8] at (6.5,11) {$\coker(\resm)$};

\draw[Brown, line width=1.5pt] (10,10)--(10,19.8);

\node[circle, Periwinkle,fill,scale=0.8] at (11,9) {};
\draw[Periwinkle, line width=1.5pt] (11,0)--(11,9);
\draw[NavyBlue, line width=1.5pt] (10,0)--(10,10);
%\node[scale=0.8] at (13,9) {$\ker(\trm)$};

\node[circle, draw, fill=white, scale=0.4] at (10,10) {};

\node[above left, scale=0.8] at (20,10) {$x\cdot 1$};
\node[above left, rotate=90, scale=0.8] at (10,20) {$y\cdot \sigma$};

%\draw [decorate,decoration={brace,amplitude=7pt},xshift=-6pt,yshift=0pt, thick]
%(10,0) -- (10,9) node [black,midway,xshift=-1.2cm, scale=0.8] 
%{$\coker(\trm)$};
%\draw [decorate,decoration={brace,amplitude=7pt,mirror},xshift=6pt,yshift=0pt, thick]
%(10,11) -- (10,20) node [black,midway,xshift=1cm,scale=0.8] 
%{$\ker(\resm)$};
%\draw (8,20.5)--(9,18);
%\draw (12,-0.5)--(11,2);
%\node[scale=0.8] at (12,-1) {$\ker(\trm)_Q$};
%\node[scale=0.8] at (8,21) {$\M(Q/e)^Q/\im(\res)$};
\end{tikzpicture}
\caption{$RO(Q)$-graded coefficients of $\hm$.}
\label{Figure introduction}
\end{figure}
\end{theorem*}

\subsection*{$\ha^Q_\star$-module structure}
The category of $Q$-Mackey functors has a symmetric monoidal structure and commutative monoids with respect to this structure are called \emph{Green functors}. If $\M$ is a Green functor, the $Q$-spectrum $\hm$ is a (naive) commutative ring $Q$-spectrum and its homotopy groups form an $RO(Q)$-graded commutative ring. The most fundamental example of a Green functor is the \emph{Burnside Mackey functor} $\A$.

Every $Q$-Mackey functor $\M$ is a module over $\A$. Since taking the Eilenberg-MacLane spectrum is a lax monoidal functor (see \cite[Chapter 8]{MR1230773}), the spectrum $\hm$ is a module over $\ha$ and $\hm^Q_\star$ is a module over the $RO(Q)$-graded commutative ring $\ha^Q_\star$.

We describe $\ha^Q_\star$ in Section \ref{sec Examples}. Its multiplicative structure is determined by two elements - $a$ and $u$. The first is an Euler class associated to the inclusion $S^0\to S^\sigma$, whereas the second is the generator of $\ha^Q_{2-2\sigma}$ and corresponds to the generator of $\ha^Q_2(S^{2\sigma})$.

The action of $\ha^Q_\star$ on $\hm^Q_\star$ may be informally stated as follows:
\begin{observation*}
The $\ha^Q_\star$-module structure of $\hm^Q_\star$ is determined by the action of 3 elements - $a$, $u$ and $\omega$, where the latter is the class of $Q/e$ in the Burnside ring.
\end{observation*}
The action by $a$ on $\hm^Q_\star$ may be easily derived from the cofibre sequence
\[
Q_+\to S^0\to S^\sigma.
\]
However, a description of the action of $u$ requires more work. The detailed analysis of this action is possible due to the Tate diagram and the connection between $\hm$ and its Borel completion. The details are given in Section \ref{sec multiplication by u}. The action of the element $\omega$ is discussed in Section \ref{sec comm}.
\subsection*{$RO(Q)$-graded ring structure}
Finally, we describe the multiplicative structure of $\hm$ when $\M$ is a Green functor.

In Sections \ref{sec multiplication by u} and \ref{sec comm} we show that most of this structure may be derived from the $\ha^Q_\star$-module structure. The first issue that we encounter is the graded commutativity. The sign rule for commutativity in $RO(Q)$-graded rings involves units in the Burnside ring - thus not only $-1$, but also $1-\omega$, where $\omega$ is the class of $Q/e$. For precise statement, see Observation \ref{obs RO(Q)-graded ring}.

In Section \ref{sec comm} we show the following Theorem:
\begin{theorem*}[{\ref{thm commutativity}}]
If $\M$ is a Green functor then $\hm^Q_\star$ is a strictly commutative ring, i.e., all signs coming from the graded commutativity rule are trivial.
\end{theorem*}

Finally, at the end of Section \ref{sec comm} in Observation \ref{obs multiplicative structure}, we give a recipe for describing the multiplicative structure of $\hm^Q_\star$ for any Green functor $\M$. 

\begin{observation*}
If $\M$ is a $Q$-Green functor then the multiplicative structure of $\hm^Q_\star$ is fully determined by its $\ha^Q_\star$-module structure and relations between elements of degrees $1-\sigma$, $\sigma-1$, $3-3\sigma$ and $3\sigma-3$. These relations may be derived from the induced map of $\hm\to\hm^h$.
\end{observation*}

The procedure of obtaining multiplicative structure is illustrated with a wide array of examples in Sections \ref{sec Examples} and \ref{sec Further examples}.

\subsection*{Contribution of the paper and related work}
Computations of coefficients of Eilenberg-MacLane spectra over $Q$ already have a long history. Based on unpublished work of Stong, Lewis computed $\ha^Q_\star$, where $\underline{\mathbb{A}}$ is the Burnside Mackey functor for $Q$ in \cite[Section 2]{MR979507}. Calculations for the constant Mackey functor $\underline{\mathbb{F}}_2$ by Caruso can be found in \cite{MR1684248} and by Hu-Kriz in \cite[Proposition 6.2]{MR1808224}. The computations for the constant Mackey functor $\underline{\mathbb{Z}}$ may be found in Dugger's work \cite[Appendix B]{MR2240234}.

The full $RO(Q)$-graded Mackey functor valued coefficients of $\hm$ for any Mackey functor $\M$ was given by Ferland in \cite{MR2699528} and may also be found in Ferland-Lewis \cite[Chapter 8]{MR2025457}. Their computations are based on the cofibre sequence
\[
Q_+\to S^0\to S^\sigma.
\] This sequence allows one to describe attaching maps of free cells, and thus to compute the homology of any representation sphere. Therefore this method may be described as the "cell method". While being intuitive and allowing the computation of the $RO(Q)$-graded abelian group structure, this approach does not provide efficient methods of describing the multiplicative structure.  We note that Ferland’s computations work for all cyclic groups of prime order, but in this paper we will restrict our attention to the group of order $2$. From Ferland's work one can derive the $RO(Q)$-graded abelian group structure and $\ha^Q_\star$-module structure of $\hm^Q_\star$.

The aim of this paper is to show the computational strength of the Tate diagram on the example of computations of coefficients of equivariant Eilenberg-MacLane spectra. As demonstrated by Greenlees in his computations of $\hz^Q_\star$ in \cite{FourApproaches}, the method based on the Tate diagram gives not only a good insight into the abelian group structure, but also into the multiplicative structure. In this paper we generalise Greenlees's results to all $Q$-Mackey and Green functors. Therefore the main contribution of this paper is the description of the multiplicative structure of $\hm^Q_\star$ and showing that the Tate diagram gives an algorithmic way to describe this structure.

The other strong computational technique in equivariant homotopy theory is the \emph{slice spectral sequence} of Hill-Hopkins-Ravenel \cite{MR3505179}. The idea behind it is a decomposition of a $G$-spectrum into a tower similar to the Postnikov tower, but retaining more equivariant information. However, one of the challenges that one may face while using the slice spectral sequence is to determine the filtration quotients. Therefore it needs additional techniques to support the calculations.  An example of this is the fact that for any $Q$-spectrum $X$ the starting input for the slice spectral sequence is given by the spectrum $\Sigma^{-1}H\underline{\pi_{-1}}(X)$, the desuspension of the Eilenberg-MacLane spectrum of $\underline{\pi_{-1}}(X)$. In this paper we demonstrate that the computations based on the Tate diagram for simple equivariant objects, such as $\hm$, may be carried out in a neat and algorithmic way and thus provide a useful method of doing auxiliary calculations for the slice spectral sequence.

Additionally, it is worth noting that equivariant homotopy theory has a rather small repository of computational examples and is in a constant need for developing methods of doing calculations. Therefore an auxiliary contribution of this paper is 
adding another piece to this development.

We note here that the "full information" is given by the \emph{Mackey functor valued} $RO(Q)$-graded coefficients $\hm^\bullet_\star$, as given for example in \cite{MR979507}. However, the main computational effort is computing the $Q/Q$ level of the $RO(Q)$-graded Mackey functor $\hm^\bullet_\star$ and this is where the Tate method gives a hand. Therefore we focus our attention in this paper on computing $\hm^Q_\star$. The rest of the Mackey functor structure of $\hm^\bullet_\star$ may be easily deduced.

\subsection{Notation and conventions}
\label{ssection Notation and conventions}
Throughout the whole paper $Q$ denotes the group of order $2$ and $\gamma$ its non-trivial element. We denote the norm element of the group ring $\mathbb{Z}[Q]$ by $N$, i.e., $N=1+\gamma$. The Burnside ring of $Q$ is denoted by $\mathbb{A}(Q)$ and we write $\omega$ for the the class of $Q/e$ in this ring.

If $M$ is a $\mathbb{Z}[Q]$-module, we denote its $Q$-fixed points by $M^Q$. We also put
\[M_Q=\frac{M}{(1-\gamma)M}.\]

We denote $N$-torsion elements of $M$ by $\prescript{}{N}M$, i.e., $\prescript{}{N}M=\{x\in M\:|\:Nx=0\}$. Two important $\mathbb{Z}[Q]$-modules are:
\begin{itemize}
\item $\mathbb{Z}$ - the integers with trivial $Q$-action;
\item $\tilde{\mathbb{Z}}$ - the integers with sign action.
\end{itemize}

We denote the $n$-th Tate cohomology group of $Q$ with coefficients in $M$ by $\hat{H}^n(Q;M)$. These can be computed to be: 
\[
\hat{H}^n(Q;M)=
\begin{cases}
M^Q/NM &\textrm{if }n\in\mathbb{Z}\textrm{ even} \\
\prescript{}{N}M/(\gamma-1)M&\textrm{if }n\in\mathbb{Z}\textrm{ odd}
\end{cases}
\]
For details and a precise definition see \cite[Definition 6.2.4]{MR1269324} or \cite[Chapter VI]{MR672956}.

Underlined capital letters are used to denote Mackey functors. If we need to show the structure of some particular Mackey functor we use the Lewis diagram:
\[
\mackey{\M(Q/Q)}{\M(Q/e)}{\res_{\M}}{\tr_{\M}}.
\]
The map $\M(Q/Q)\to\M(Q/e)$ is the \emph{restriction map} and will be denoted by $\res_{\M}$. The map $\M(Q/e)\to\M(Q/Q)$ is called the \emph{transfer map} and will be denoted by $\tr_{\M}$. We will drop the subscripts if it is clear from the context which Mackey functor the notation refers to. For brevity we use $V$ to denote the $\mathbb{Z}[Q]$-module $M(Q/e)$.

Throughout this whole paper we work in the category of genuine $Q$-spectra. By "commutative ring $Q$-spectrum" we mean an $E_\infty$-ring in $Q$-spectra. In the literature these are also called naive commutative ring $Q$-spectra.

If $X$ is a $Q$-spectrum we put $X^Q_\star=\pi^Q_\star(X)$. The five-pointed star is used to indicate the grading over $RO(Q)$ - the representation ring of $Q$. This means that our grading is "two-dimensional" - all representations of $Q$ may be written in the form $x\mathbb{R}+y\sigma$, where $x,y$ are integers, $\mathbb{R}$ is a real trivial one-dimensional representation and $\sigma$ is a real sign representation. We will abbreviate $x\mathbb{R}+y\sigma$ to $x+y\sigma$. We sometimes refer to $x$ in the gradation as a \emph{fixed degree} and to $y$ as a \emph{twisted degree}.
By the \emph{antidiagonal} we mean the line $y=-x$. Whenever we want to restrict our attention to $\mathbb{Z}$-grading, we will use an asterisk $\ast$ instead.

We put $X^e_\star:=\pi^Q_\star(F(Q_+,X))$. Note that by the induction-restriction adjunction $X^e_{x+y\sigma}$ is the same as the $\mathbb{Z}[Q]$-module $\pi_{x+y}(\res^Q_e(X))$, which we abbreviate as $\pi_{x+y}(X)$.  A map $\phi\colon X\to Y$ that induces an isomorphism $X^e_\ast\cong Y^e_\ast$ will be called a \emph{nonequivariant equivalence}. 

Note that in some papers a different grading convention for $X_\star^Q$ is used, e.g., by Dugger in \cite{MR2240234}. Our grading is related to this by $\mathbb{R}^{x+y,y}=x+y\sigma$.

\subsection{Acknowledgments}
I would like to thank John Greenlees for the idea and continuous supervision of this work. I am also very grateful to Luca Pol and Jordan Williamson for numerous discussions, comments and corrections. The final version of this paper owes its readability to their countless suggestions and patient reading of previous versions.

This paper also significantly improved as a result of the comments and suggestions made by the anonymous referee, to whom my thanks are also due.
\section{The Tate method}
\label{sec The Tate method}
In this section we recall the Tate diagram method developed in \cite{MR1230773}.

 Let $EQ$ be a free contractible $Q$-space and let us define the \emph{isotropy separation sequence} to be the cofibre sequence:

\[
\label{EQ}
\begin{tikzcd}
EQ_+\rar & S^0\rar &\widetilde{EQ}.
\end{tikzcd}
\]

Let $X$ be a $Q$-spectrum. Let $\epsilon$ be the map \[\epsilon\colon X\cong F(S^0, X)\to F(EQ_+,X).\] By smashing $\epsilon$ with the isotropy separation sequence we obtain the following diagram:
\[
\begin{tikzcd}
\label{diagram Tate intro}
EQ_+\wedge X\rar\dar["EQ_+\wedge\epsilon"]&
X\rar\dar["\epsilon"]&
\widetilde{EQ}\wedge X\dar["\widetilde{EQ}\wedge\epsilon"]\\
EQ_+\wedge F(EQ_+, X)\rar&
F(EQ_+, X)\rar&
\widetilde{EQ}\wedge F(EQ_+,X).
\end{tikzcd}
\tag{$\ast$}
\]
The left-hand vertical map $EQ_+\wedge\epsilon$ is an equivalence by the following \cite[Proposition 1.1]{MR1230773}:
\begin{lemma}
\label{prop_GM 1.2}
Let $\phi\colon X\to Y$ be a map of $Q$-spectra which is a nonequivariant equivalence. Then the induced maps
\[
\begin{array}{c}
\phi\wedge EQ_+\colon X\wedge EQ_+\to Y\wedge EQ_+ \\
F(EQ_+,\phi )\colon F(EQ_+,X)\to F(EQ_+,Y)
\end{array}
\]
are $Q$-equivalences.
\end{lemma}
Since the map $\epsilon$ is a nonequivariant equivalence, the map $EQ_+\wedge\epsilon$ is a $Q$-equivalence. Therefore the right-hand square is a homotopy pullback square.

We define
\begin{align*}
X^h&:=F\left(EQ_+,X\right), \\
X_h&:=EQ_+\wedge X, \\
X^t&:=F\left(EQ_+,X\right)\wedge\widetilde{EQ}=X^h\wedge\widetilde{EQ}, \\
X^\Phi&:=\widetilde{EQ}\wedge X.
\end{align*}
We also put $X_{hQ}:=(X_h)^Q$, respectively for $X^{hQ}$, $X^{\Phi Q}$ and $X^{tQ}$.

After renaming entries in the diagram \eqref{diagram Tate intro} we obtain the following commutative diagram, called the \emph{Tate diagram}:
\[
\label{Tate_square}
\begin{tikzcd}
X_h \rar["f"]\dar & X \rar["g"]\dar["\epsilon"] & X^\Phi\dar["\epsilon_t"] \\
X_h \rar["\nu "] & X^h \rar["\theta"] & X^t.
\end{tikzcd}
\]

The right-hand square in the Tate diagram is a homotopy pullback square of $Q$-spectra and is known as the \emph{Tate square}. If additionally $X$ is a ring $Q$-spectrum, all of the corners of the Tate square are also ring $Q$-spectra and the square is a homotopy pullback of ring $Q$-spectra. The analogous statement holds for $X$-module $Q$-spectra (see \cite[Proposition 3.5]{MR1230773}).

We will need the following fact about $X_h$:
\begin{prop}
\label{prop X_h is a module over X^h}
If $X$ is a ring $Q$-spectrum, $X_h$ is a module over $X^h$ in the homotopy category.
\end{prop}
\begin{proof}
By Proposition \ref{prop_GM 1.2} the map $EQ_+\wedge\epsilon$ is a $Q$-equivalence, so it has an inverse in the homotopy category. Denote this inverse by $\bar{\epsilon}$. Let $\mu\colon X\wedge X\to X$ be the multiplication in $X$. The $X^h$-module structure is given by the following composite:
\[
\begin{tikzcd}[column sep=large]
F(EQ_+,X)\wedge (EQ_+\wedge X)
\rar["\bar{\epsilon}\wedge 1"] &
X\wedge EQ_+\wedge X\rar["twist"]&
EQ_+\wedge X\wedge X
\rar["EQ_+\wedge\mu"] &
EQ_+\wedge X.
\end{tikzcd}
\]
\end{proof}
Now we will describe an important multiplicative property of the spectra $X^\Phi$ and $X^t$. Let $a$ be the element of $\pi^Q_{-\sigma}(S^0)$ corresponding to the inclusion $S^0\to S^\sigma$. The map $a\wedge 1\colon X\to S^\sigma\wedge X$ gives a multiplication by $a$ in $X$. If $a$ acts as an isomorphism on $X^Q_\star$ we say that $X$ is \emph{$a$-periodic}.
\begin{lemma}
\label{lemma_a periodicity of geometric fp and tate}
If $X$ is a $Q$-spectrum, then
\[
\pi^Q_\star(X\wedge\widetilde{EQ})= a^{-1}X^Q_\star.
\]
So the spectrum $X\wedge\widetilde{EQ}$ is $a$-periodic. In particular $X^\Phi$ and $X^t$ are $a$-periodic.
\end{lemma}
\begin{proof}
The $Q$-space $\widetilde{EQ}$ has a model $S^{\infty\sigma}$. So $X\wedge\widetilde{EQ}$ may be seen as a homotopy colimit of the sequence:
\[
\begin{tikzcd}
X \rar["a"] & S^{\sigma}\wedge X\rar["a"] & S^{2\sigma}\wedge X\rar["a"]&\ldots.
\end{tikzcd}
\]
After applying $\pi^Q_\star$ we obtain the following sequence:
\[
\begin{tikzcd}
X^Q_\star\rar["a"]&X^Q_{\star-\sigma}\rar["a"]&X^Q_{\star-2\sigma}\rar["a"]&\ldots
\end{tikzcd}
\]
from which we get an identification
\[
\pi^Q_\star(X\wedge\widetilde{EQ})\cong a^{-1}X^Q_\star.
\]
Thus $a$ acts as an isomorphism on $\pi^Q_\star(X\wedge\widetilde{EQ})$.
\end{proof}

Let $\M$ be a $Q$-Mackey functor. To compute $\hm^Q_\star$ we will use the following method, described in \cite{FourApproaches}: 
\begin{enumerate}
\item Firstly we compute $\hm^{hQ}_\star$ and $\ho{\hm}{\star}$ using the homotopy fixed point and homotopy orbit spectral sequences - details are in Section \ref{section hfp, ho and tate}.
\item Using Lemma \ref{lemma_a periodicity of geometric fp and tate} we deduce $\hm^{tQ}_\star$ from $\hm^{hQ}_\star$ by inverting $a$, since $\hm^t=\hm^h\wedge\widetilde{EQ}$.
\item We infer $\hm^{\Phi Q}_\star$  from $\hm^{tQ}_\star$. Since both theories are $a$-periodic, we need only to compute $\hm^{\Phi Q}_n$ for $n\in\mathbb{Z}$. We calculate them using the fact that fibres of $\epsilon$ and $\epsilon_t$ in the Tate diagram are equivalent.
\item Finally we deduce $\hm^Q_\star$ from the Tate diagram.
\end{enumerate}
\section{Structure of $\hm^Q_{\ast\sigma}$}
\label{sec structure of hm y axis}
From now on let $\M$ be a Mackey functor. We start with a description of the subgroup $\hm^Q_{\ast\sigma}$ with fixed degree $0$. In this case the computations follow from the cofibre sequence
\[
\label{cofibre}
\begin{tikzcd}
Q_+\rar & S^0 \rar & S^\sigma.
\end{tikzcd}\tag{$\dagger$}
\]
From this sequence we can deduce the following lemma, which gives a complete description of the multiplication by the element $a\in\pi^Q_{-\sigma}(S^0)$ on Eilenberg-MacLane spectra:
%
%Lemma 3.1
%
\begin{lemma}
\label{lemma multiplication by a}
The map $a\colon \hm^Q_{x+y\sigma}\to \hm^Q_{x+(y-1)\sigma}$ is:
\begin{enumerate}
\item a monomorphism if $x-1=-y$;
\item an epimorphism if $x=-y$;
\item an isomorphism otherwise.
\end{enumerate}
\end{lemma}
\begin{proof}
After smashing \eqref{cofibre} with $\hm$ and applying $\pi^Q_{x+y\sigma}(-)$ we obtain an exact sequence:
\[
\begin{tikzcd}
\hm^e_{x+y\sigma} \rar & \hm_{x+y\sigma}^Q \rar["a"] & \hm_{x+(y-1)\sigma}^Q \rar & \hm^e_{(x-1)+y\sigma}.
\end{tikzcd}
\]
Note that the outer groups are of the form $\hm^e_{x+y\sigma}\cong\pi_{x+y}(\res^Q_e\hm)$. The spectrum $\res^Q_e\hm$ is the Eilenberg-MacLane spectrum associated to the abelian group $V=\M(Q/e)$. Thus $\hm^e_{x+y\sigma}=0$ unless $x=-y$ and the lemma follows.
\end{proof}
\begin{rem}
The same result can be found in \cite[Lemma 8.7(b)]{MR2025457}.
\end{rem}
Lemma \ref{lemma multiplication by a} simplifies the calculations of the groups $\hm^Q_{y\sigma}$ for $y\in\mathbb{Z}$. These are also derived from the cofibre sequence \eqref{cofibre}.

%
%Proposition 3.2
%
\begin{prop}
\label{prop groups hmq for x=0}
The groups $\hm^Q_{y\sigma}$ are given by:
\[
\hm^Q_{y\sigma}=
\begin{cases}
\ker(\resm) & \textrm{if }y>0 \\
\coker(\trm) & \textrm{if }y<0 \\
\M(Q/Q) & \textrm{if }y=0.
\end{cases}
\]
Moreover, the multiplication $a\colon\hm^Q_{\sigma}\to\hm^Q_{0}$ is the inclusion of $\ker(\resm)$ in $M(Q/Q)$ and $a\colon\hm^Q_{0}\to\hm^Q_{-\sigma}$ is the projection onto $\coker(\trm)$.
\end{prop}
\begin{proof}
The case $y=0$ is the definition of an Eilenberg-MacLane spectrum of a Mackey functor. So we need to prove the proposition for $y\neq 0$.

By Lemma \ref{lemma multiplication by a} it is enough to prove the claim for $y=1$ and $y=-1$. If we smash the cofibre sequence \eqref{cofibre} with $\hm$ and apply $\pi^Q_0(-)$ we obtain an exact sequence
\[
\begin{tikzcd}
H\M^e_0 \rar["\trm"] & H\M^Q_0 \rar["a"] & H\M^Q_{-\sigma} \rar & 0.
\end{tikzcd}
\]
Thus $\hm^Q_{-\sigma}=\coker(\trm)$ and the multiplication by $a$ is the projection. Analogously, by applying $[-,\hm]^Q$ to \eqref{cofibre} we get that $\hm^Q_{\sigma}=\ker(\resm)$.
\end{proof}
\section{Homotopy fixed points, homotopy orbits and the Tate spectrum of $\hm$}
\label{section hfp, ho and tate}
In this section we compute the coefficients of the $Q$-spectra $\hm^h$, $\hm_h$ and $\hm^t$. These are given in Proposition \ref{prop coefficients of hfp and ho} and Proposition \ref{prop coefficients of Tate}. However, firstly we need to give a brief description of the homotopy fixed point and homotopy orbit spectral sequences.
\subsection{Homotopy fixed point and homotopy orbit spectral sequences}
\label{subsec hfp ho spectral sequence}
The space $EQ$ used in the definitions of $Q$-spectra $\hm^h$, $\hm_h$ and $\hm^t$ (see Section \ref{sec The Tate method}) has a very convenient model of the form $S(\infty\sigma)$. This space has a filtration by skeleta
\[
\label{filtration}
S(\sigma )\subset S(2\sigma )\subset S(3\sigma)\subset\ldots
\tag{$\ast$}
\]
and consequently, if $X$ is a $Q$-spectrum, we obtain a filtration of the spectra $X_h$ and $X^h$. The spectral sequences associated to these filtrations take the form
\[
E^2_{pq}=H_p(Q;\pi_q(X))\Rightarrow \pi_{p+q}(X_{hQ})
\]
and
\[
E_2^{pq}=H^p(Q; \pi_{-q}(X))\Rightarrow \pi_{-(p+q)}(X^{hQ}).
\]
The first spectral sequence is called \emph{the homotopy orbit spectral sequence} and the second is known as \emph{the homotopy fixed point spectral sequence}. Details of the construction may be found in \cite[Chapter 10]{MR1230773}.

Both spectral sequences can be used to compute the $RO(Q)$-graded coefficients of $X_h$ and $X^h$. To this end we note that
\[
X^{hQ}_{x+y\sigma}=[S^{x+y\sigma}, F(EQ_+,X)]^Q\cong [S^x, F(EQ_+, F(S^{y\sigma}, X))]^Q= \pi_x(F(S^{y\sigma},X)^{hQ}).  
\]
One may argue similarly for the homotopy orbits, and thus we can arrange the homotopy orbit and homotopy fixed point spectral sequences to be trigraded:
\[
E^2_{pq}(y)=H_p\left(Q; \pi_q\left(F(S^{y\sigma},X)\right)\right)\Rightarrow \pi_{p+q}(F(S^{y\sigma},X)_{hQ})
\]
and
\[
E_2^{pq}(y)=H^p\left(Q;\pi_{-q}\left(F(S^{y\sigma},X)\right)\right)\Rightarrow \pi_{-(p+q)}(F((S^{y\sigma}),X)^{hQ}).
\]
Note that even though the spectral sequences are trigraded, the differentials live on the "layers" corresponding to the single value of $y$. If $X$ is a ring $Q$-spectrum then the pairings
\[
S^{y_1\sigma}\wedge S^{y_2\sigma}\to S^{(y_1+y_2)\sigma}
\]
give the trigraded homotopy fixed point spectral sequence a multiplicative structure.

It will prove useful to have a description of the $E_1$-page. Since we will only use it for the homotopy fixed point spectral sequence, we omit here the description of the $E^1$-page of the homotopy orbit spectral sequence.

The filtration \eqref{filtration} gives a filtration of $X^{hQ}\cong F_Q(EQ_+,X)$, the spectrum of $Q$-equivariant maps from $EQ_+$ to $X$. Therefore the $E_1$-page has the form
\[
E^{pq}_1=\pi_{-q}(X).
\]
However, there is a more useful description of $E_1^{pq}$ given in \cite[Chapter 9]{MR1230773} which allows us also to determine differentials. 
\begin{prop}
\label{prop isomorphism from GM}
There is an isomorphism
\[
E^{pq}_1=[Q_+\wedge S^p,\Sigma^{p+q}X]^Q\cong\Hom_{\mathbb{Z}[Q]}\left(H_p(Q_+\wedge S^p),X^e_{-q}\right).
\]
\end{prop}
The differential 
\[d_1:E^{p,q}\to E^{p+1,q}\]
is induced by a map $\partial_*: H_{p-1}(Q_+\wedge S^p)\to H_p(Q_+\wedge S^{p-1}) $, which is further induced in homology by the \emph{geometric boundary}:
\[
\partial: Q_+\wedge S^p \to \Sigma S\left( (p-1)\sigma \right)_+\to \Sigma (Q_+\wedge S^{p-1}).
\]
Note that the complex
\[
\begin{tikzcd}
\ldots \rar & H_{p+1}(Q_+\wedge S^{p+1}) \rar["\partial_*"] & H_p(Q_+\wedge S^p) \rar["\partial_*"] & H_{p-1}(Q_+\wedge S^{p-1}) \rar &\ldots 
\end{tikzcd}
\]
is the cellular complex of $S(\infty\sigma)$ and thus it is exact ($S(\infty\sigma)$ is contractible) with differentials given by the degrees of attaching maps. Therefore, since $H_p(Q_+\wedge S^p)\cong\mathbb{Z}[Q]$, the complex above gives us a $2$-periodic $\mathbb{Z}[Q]$-resolution of $\mathbb{Z}$. Thus along the $q$-th row on the $E_1$-page we have the complex computing the group cohomology with coefficients in $X^e_{-q}$. This concludes the description of the $E_1$-page.
\subsection{Calculations for Eilenberg-MacLane spectra}

Now we specialise to the case $X=\hm$. Recall the notation $V:=\M(Q/e)$.
\begin{prop}
\label{prop coefficients of hfp and ho}
The $RO(Q)$-graded coefficients of $\hm_h$ and $\hm^h$ are given by:
\begin{itemize}
\item $\left(\hm^Q_{hQ}\right)_{(y\sigma-y)+p}= H_p(Q;H^y(S^{y\sigma},V))$
\item $\hm^{hQ}_{(y\sigma-y)-p}= H^p(Q; H^y(S^{y\sigma},V))$.
\end{itemize}
\end{prop}
\begin{proof}
The proposition follows from the fact that the trigraded homotopy orbit and homotopy fixed point spectral sequences collapse on the second page.
The coefficients of the group homology/cohomology on the second page are given by
\[
\pi_p(F(S^{y\sigma},\hm))\cong H^{-p}(S^{y\sigma}, V).
\]
The (reduced) singular cohomology $H^{-p}(S^{y\sigma}, V)$ is $0$ unless $-p=y$. Thus for every  $y$ the $E_2(y)$-page consists of one row and all differentials are $0$.
\end{proof}
\begin{rem}
\label{rem cohomology of S^nsigma}
Note that since $Q$ acts on $H^y(S^{y\sigma},\mathbb{Z})$ by the degree of $\gamma$ as a map, it is isomorphic to $\mathbb{Z}$ if $y$ is even and $\tilde{\mathbb{Z}}$ if $y$ is odd. So by the Universal Coefficient Theorem $H^y(S^{y\sigma},V)$ is isomorphic to $V$ if $y$ is even and $\tilde{V}:=\tilde{\mathbb{Z}}\otimes V$ if $y$ is odd.
\end{rem}

The projective $\mathbb{Z}[Q]$-resolution of $\tilde{\mathbb{Z}}$ is given by:
\[
\begin{tikzcd}
\ldots\rar & \mathbb{Z}[Q]\rar["1-\gamma"] & \mathbb{Z}[Q]\rar["1+\gamma"] & \mathbb{Z}[Q]\rar["\tilde{\epsilon}"]&\tilde{\mathbb{Z}}\rar & 0.
\end{tikzcd}
\]
Here $\tilde{\epsilon}$ is the ring map defined by $\tilde{\epsilon}(\gamma)=-1$.

Recall the 2-periodic $\mathbb{Z}[Q]$-resolution of $\mathbb{Z}$:
\[
\begin{tikzcd}
\ldots\rar & \mathbb{Z}[Q]\rar["1+\gamma"] & \mathbb{Z}[Q]\rar["1-\gamma"] & \mathbb{Z}[Q]\rar["\epsilon"]&\mathbb{Z}\rar & 0.
\end{tikzcd}
\]
Therefore from the resolution of $\tilde{\mathbb{Z}}$ we can see that there are isomorphisms $H^i(Q;\tilde{V})\cong H^{i+1}(Q;V)$ and $H_i(Q;\tilde{V})\cong H_{i+1}(Q;V)$ for $i\geq 1$. This gives us four potentially non-zero values for $(\hm_{hQ})_\star$ and four for $\hm^{hQ}_\star$, as depicted in Figures \ref{picture_homotopyorbits} and \ref{picture_homotopy fixed points}.
\begin{figure}[ht]
\begin{minipage}[t]{0.4\textwidth}
\noindent
\begin{tikzpicture}[scale=0.6]
\draw[->, thick] (0,5)--(10,5) ;
\draw[->, thick] (5,0)--(5,10) ;
\draw[help lines, thin, lgray] (0,0) grid (10,10);
\foreach \x in {0,2,4}
%Grube kropki
\node[circle,fill,scale=0.6] at (5-\x,5+\x) {};
\foreach \x in {2,4}
\node[circle,fill,scale=0.6] at (5+\x,5-\x) {};

%Diamenty
\foreach \x in {1,3,5}
\node[diamond,fill,scale=0.6] at (5-\x,5+\x) {};
\foreach \x in {1,3,5}
\node[diamond,fill,scale=0.6] at (5+\x,5-\x) {};

%Chude pełne kropki
\foreach \x in {0,...,10}
\node[circle,scale=0.4,fill] at (10,\x) {};
\foreach \x in {2,...,10}
\node[circle,scale=0.4,fill] at (8,\x) {};
\foreach \x in {4,...,10}
\node[circle,scale=0.4,fill] at (6,\x) {};
\foreach \x in {6,...,10}
\node[circle,scale=0.4,fill] at (4,\x) {};
\foreach \x in {8,...,10}
\node[circle,scale=0.4,fill] at (2,\x) {};

%Puste kropki
\foreach \x in {9,10}
\draw (1,\x) circle (3pt);
\foreach \x in {8,...,10}
\draw (3,\x) circle (3pt);
\foreach \x in {6,...,10}
\draw (5,\x) circle (3pt);
\foreach \x in {4,...,10}
\draw (7,\x) circle (3pt);
\foreach \x in {2,...,10}
\draw (9,\x) circle (3pt);

\node[above left] at (10,5) {$x\cdot 1$};
\node[below right] at (5,10) {$y\cdot \sigma$};

\node[scale=1.5] at (5,11) {$(H\underline{M}_{hQ})_{x+y\sigma}$};

%Legenda
\node at (0.5,-1) {Key:};
\node[circle, fill, scale=0.6] at (2,-1) {};
\node at (4,-1) {$V/(1-\gamma)V$};
\node[diamond, fill, scale=0.6] at (2,-2) {};
\node at (4,-2) {$V/(1+\gamma)V$}; 
\node[circle, fill,scale=0.4] at (6.4,-1) {};
\node at (8.5,-1) {$V^Q/(1+\gamma)V$};
\draw (6.4,-2) circle (3pt);
\node at (8.5,-2) {$\prescript{}{N}V/(1-\gamma)V$};
\end{tikzpicture}
\caption{Coefficients of $\hm_h$}
\label{picture_homotopyorbits}
\end{minipage}
\hfill
\begin{minipage}[t]{0.4\textwidth}

\noindent
\begin{tikzpicture}[scale=0.6]
\draw[->, thick] (0,5)--(10,5) ;
\draw[->, thick] (5,0)--(5,10) ;
\draw[help lines, thin, lgray] (0,0) grid (10,10);
%Grube kropki
\foreach \x in {0,2,4}
\node[circle,fill,scale=0.6] at (5-\x,5+\x) {};
\foreach \x in {2,4}
\node[circle,fill,scale=0.6] at (5+\x,5-\x) {};

%Diamenciki
\foreach \x in {1,3,5}
\node[diamond,fill,scale=0.6] at (5-\x,5+\x) {};
\foreach \x in {1,3,5}
\node[diamond,fill,scale=0.6] at (5+\x,5-\x) {};

%Chude pełne kropki
\foreach \x in {0,...,8}
\node[circle,scale=0.4,fill] at (1,\x) {};
\foreach \x in {0,...,6}
\node[circle,scale=0.4,fill] at (3,\x) {};
\foreach \x in {0,...,4}
\node[circle,scale=0.4,fill] at (5,\x) {};
\foreach \x in {0,...,2}
\node[circle,scale=0.4,fill] at (7,\x) {};
\node[circle,scale=0.4,fill] at (9,0) {};

%Puste kropki
\foreach \x in {0,...,9}
\draw (0,\x) circle (3pt);
\foreach \x in {0,...,7}
\draw (2,\x) circle (3pt);
\foreach \x in {0,...,5}
\draw (4,\x) circle (3pt);
\foreach \x in {0,...,3}
\draw (6,\x) circle (3pt);
\foreach \x in {0,1}
\draw (8,\x) circle (3pt);

\node[above left] at (10,5) {$x\cdot 1$};
\node[below right] at (5,10) {$y\cdot \sigma$};

\node[scale=1.5] at (5,11) {$H\underline{M}^{hQ}_{x+y\sigma}$};

%Legenda
\node at (0.5,-1) {Key:};
\node[circle, fill, scale=0.6] at (2,-1) {};
\node at (3,-1) {$V^Q$};
\node[circle, fill,scale=0.4] at (5,-1) {};
\node at (7.2,-1) {$V^Q/(1+\gamma)V$};
\node[diamond, fill, scale=0.6] at (2,-2) {};
\node at (3,-2) {$\prescript{}{N}V$}; 
\draw (5,-2) circle (3pt);
\node at (7.2,-2) {$\prescript{}{N}V/(1-\gamma)V$};

\end{tikzpicture}
\caption{Coefficients of $\hm^h$}
\label{picture_homotopy fixed points}
\end{minipage}
\end{figure}

\begin{rem}
Note that from Figures \ref{picture_homotopyorbits} and \ref{picture_homotopy fixed points} it is easy to see that both $\ho{\hm}{\star}$ and $\hm^{hQ}_\star$ are $(2-2\sigma)$-periodic. We will attribute this phenomenon to the multiplication by the generator of $\ha^Q_{2-2\sigma}$ in Section \ref{sec multiplication by u}. 
\end{rem}
Before proceeding to the description of $\hm^{tQ}_\star$ we give a couple of observations on the structure of $\hm_h$ and $\hm^h$. 

\begin{observation}
\label{obs multiplication by a on ho and hfp}
We have that
\[
\pi^e_\ast(\hm_h)\cong\pi^e_\ast(\hm^h)\cong\pi^e_\ast(\hm).
\]
In particular, the multiplication by $a$ in $\left(\hm_{hQ}\right)_\star$ and $\hm^{hQ}_\star$ can be described in the same way as in Lemma \ref{lemma multiplication by a}.
\end{observation}
\begin{proof}
We are going to prove the observation only for $\hm_h$, as the proof for $\hm^h$ follows analogously.
Note that $\res^Q_e(EQ_+)\simeq S^0$. By the definition of non-equivariant homotopy groups we have that
\begin{align*}
\pi^e_n(\hm_h)&=[Q_+\wedge S^n,EQ_+\wedge\hm]^Q\\
&\cong [S^n,\res^Q_e(EQ_+\wedge \hm)]\\
&\cong [S^n, \res^Q_e(EQ_+)\wedge \res^Q_e\hm]\\
&\cong [S^n, S^0\wedge\res^Q_e]=\pi^e_n(\hm).
\end{align*}
Here the second isomorphism comes from the fact that the restriction functor is symmetric monoidal. The statement about the multiplication by $a$ follows now by the verbatim repetition of the proof of Lemma \ref{lemma multiplication by a}.
\end{proof}
\begin{prop}
\label{prop epsilon_0 f_0}
The map $f_0\colon V_Q=\ho{\hm}{0}\to\hm^Q_0=\M(Q/Q)$ is the map induced by $\trm$ on $V_Q$. The map $\epsilon_0\colon \M(Q/Q)=\hm^Q_0\to\hm^{hQ}_0=V^Q$ is the restriction of $\resm$ to the codomain $V^Q$. 
\end{prop}
\begin{proof}
Let $i\colon Q_+\to EQ_+$ be the inclusion of the $0$-skeleton. Note that the map $Q_+\to S^0$ extends to the filtration \eqref{filtration} of $EQ_+$ and thus it factors as
\[
\begin{tikzcd}
Q_+\rar["i"] & EQ_+\rar & S^0. 
\end{tikzcd}
\]
By smashing this sequence with $\hm$ and taking $\pi^Q_0(-)$ we obtain the following commutative diagram:
\[
\begin{tikzcd}[column sep=1.5em]
\hm^e_0 \ar[rr, "\trm"]\ar[dr] && \hm^Q_0 \\
&\left(\hm_{hQ}\right)_0\ar[ur,"f_0"']&
\end{tikzcd}
\]
By Proposition \ref{prop coefficients of hfp and ho} this diagram is isomorphic to the following:
\[
\begin{tikzcd}[column sep=small]
V \ar[rr, "\trm"]\ar[dr] && \M (Q/Q) \\
&\sfrac{V}{(1-\gamma)V} \ar[ur,"f_0"']&
\end{tikzcd}
\]
The left diagonal arrow is the canonical projection. Thus the map $f_0\colon\ho{\hm}{0}\to\hm^Q_0$ is the map induced by the transfer on $V_Q$. Note that by the properties of Mackey functors $\trm$ always factors via $V_Q$. We obtain the second assertion by using the dual argument.
\end{proof}
We now give a corollary, which will become useful in Section \ref{sec multiplication by u}.
\begin{cor}
\label{cor fysigma for y>1}
If $y\geq 1$, then the map $f_{y\sigma}\colon\left(\hm_{hQ}\right)_{y\sigma}\to\hm^Q_{y\sigma}$ is induced by the transfer.
\end{cor}
\begin{proof}
By Lemma \ref{lemma multiplication by a} and Observation \ref{obs multiplication by a on ho and hfp} we need to prove only the case $y=1$.

Consider the following commutative diagram:
\[
\begin{tikzcd}
\left(\hm_{hQ}\right)_\sigma\dar["a"]\rar["f_\sigma"]&\hm^Q_\sigma\dar["a"]\\
\left(\hm_{hQ}\right)_0\rar["f_0"]&\hm^Q_0.
\end{tikzcd}
\]

From Lemma \ref{lemma multiplication by a} and Observation \ref{obs multiplication by a on ho and hfp} we get that the vertical arrows in this diagram are inclusions. By Propositions \ref{prop groups hmq for x=0} and \ref{prop coefficients of hfp and ho} it is isomorphic to the following diagram:
\[
\begin{tikzcd}
\frac{\prescript{}{N}V}{(1-\gamma)V}\dar["a"]\rar["f_\sigma"]&\ker(\resm)\dar["a"]\\
V_Q\rar["f_0"]&M(Q/Q).
\end{tikzcd}
\]

Since the map $f_0$ is induced by the transfer in $\M$ by Proposition \ref{prop epsilon_0 f_0}, the map $f_\sigma$ is also induced by the transfer. The statement follows.
\end{proof}
%
%Proposition - coefficients of Tate
%
\begin{prop}
\label{prop coefficients of Tate}
The coefficients of the $Q$-spectrum $\hm^t$ are given by
\[
\hm^{tQ}_{x+y\sigma}=\hat{H}^{-x}(Q;V),
\]
where $\hat{H}^x(Q;V)$ denotes the $x$-th Tate cohomology of $Q$ with coefficients in $V$ (see Section \ref{ssection Notation and conventions}). Note that since $\hm^t$ is $a$-periodic, $\hm^{tQ}_{x+y\sigma}$ does not depend on $y$.
\end{prop}
\begin{proof}
By definition $X^t=X^h\wedge\widetilde{EQ}$. So by Lemma \ref{lemma_a periodicity of geometric fp and tate} we have that
\[
\hm^{tQ}_\star= a^{-1}\hm^{hQ}_\star.
\]
From this isomorphism we see in particular that $\hm^{tQ}_{x+y\sigma}=\hm^{hQ}_{x+y\sigma}$ below the antidiagonal, i.e., when $y<-x$. So if $x$ is even we obtain by Proposition \ref{prop coefficients of hfp and ho} and $a$-periodicity of $\hm^t$ that
\[
\hm^{tQ}_{x+y\sigma}\cong\hm^{tQ}_{x-(x+1)\sigma}=\hm^{hQ}_{x-(x+1)\sigma}\cong V^Q/NV.
\]
Analogously, if $x$ is odd we obtain that $\hm^{tQ}_{x+y\sigma}\cong\prescript{}{N}V/(1-\gamma)V$.
\end{proof}
\section{Geometric fixed points}
\label{sec gfp}
In this section we describe the coefficients of the $Q$-spectrum $\hm^\Phi$.
\begin{theorem}
\label{thm coefficients of geometric fixed points}
The coefficients of $\hm^\Phi$ are given by:
\[
\hm^{\Phi Q}_{x+y\sigma}=
\begin{cases}
\coker(\trm)&\textrm{if } x=0 \\
\left(\ker(\trm)\right)_Q&\textrm{if } x=1 \\
\hm^{tQ}_{x+y\sigma}=\hat{H}^x(Q;V)&\textrm{if } x\geq 2 \\
0&\textrm{otherwise}.
\end{cases}
\]
\end{theorem}
Note that the values of $\hm^{\Phi Q}_\star$ depend only on the fixed degree, which comes from Lemma \ref{lemma_a periodicity of geometric fp and tate}. We prove Theorem \ref{thm coefficients of geometric fixed points} by a series of lemmas.
%
%Lemma - 0,1 of geometric fixed points
%
\begin{lemma}
\label{lemma gfp1}
$\hm^{\Phi Q}_{y\sigma}=
\coker(\trm)$ and $\hm^{\Phi Q}_{1+y\sigma}=
\left(\ker(\trm)\right)_Q$.
\end{lemma}
\begin{proof}
By Lemma \ref{lemma_a periodicity of geometric fp and tate} we have that $\hm^{\Phi Q}_{y\sigma}\cong\hm^{\Phi Q}_0$, so we need to prove the claim only in this case. Analogously for $x=1$.
From the top row of the Tate diagram we obtain the following exact sequence:
\[
\begin{tikzcd}
\hm^Q_1\rar["g_1"]&\hm^{\Phi Q}_1\rar &\ho{\hm}{0}\rar["f_0"]&\hm^{Q}_0\rar["g_0"] &\hm^{\Phi Q}_0\rar &\ho{\hm}{-1}
\end{tikzcd}
\]
The outer groups in this sequence are zero: $\hm^Q_1$ by the definition of an Eilenberg-MacLane spectrum and $\ho{\hm}{-1}$ by Proposition \ref{prop coefficients of hfp and ho}. Thus $\hm^{\Phi Q}_{0}$ and $\hm^{\Phi Q}_1$ are respectively cokernel and kernel of the map $f_0$. Thus by Proposition \ref{prop epsilon_0 f_0} we obtain that $\coker(f_0)=\coker(\trm )$ and $\ker(f_0)=\ker(\trm)_Q$.
\end{proof}
%
%Lemma - geometric fixed points, x>2
%
\begin{lemma}
\label{lemma gfp iso to tate}
If $x\geq 2$ then $\hm^{\Phi Q}_{x+y\sigma}=\hm^{tQ}_{x+y\sigma}$.
\end{lemma}
\begin{proof}
Since both $\hm^\Phi$ and $\hm^t$ are $a$-periodic by Lemma \ref{lemma_a periodicity of geometric fp and tate}, we need only to show that if $x\geq 2$, then $\hm^{\Phi Q}_x=\hm^{tQ}_x$ (with twisted degree $0$).
Note that the fibres of $\epsilon\colon\hm\to\hm^h$ and $\epsilon_t\colon\hm^\phi\to\hm^t$ are equivalent, as they are both of the form $F(\widetilde{EQ},\hm)$. Let $F$ denotes this fibre. By applying the long exact sequence in homotopy to the fibration $F\to\hm\to\hm^h$ we get:
\[
\begin{tikzcd}
\ldots\rar&\hm^Q_{m+1}\rar["\epsilon_\ast"]&\hm^{hQ}_{m+1}\rar &F^Q_m\rar&\hm^Q_m\rar["\epsilon_\ast"]&\hm^{hQ}_m\rar&\ldots.
\end{tikzcd}
\] 
But by the definition of Eilenberg-MacLane spectrum we have that $\hm^Q_m=0$ if $m\neq 0$ and $\hm^{hQ}_m=0$ for $m\geq 1$ by Proposition \ref{prop coefficients of hfp and ho}, so we get from this exact sequence that $F^Q_m=0$ for $m\geq 1$. By applying the analogous long exact sequence to the fibration $F\to \hm^\Phi\to\hm^t$ we get that $\hm^{\Phi Q}_m=\hm^{tQ}_m$ for $m\geq 2$.
\end{proof}
%
%Lemma - geometric fixed points, x<0
%
\begin{lemma}
\label{lemma gfp 3}
$\hm^{\Phi Q}_{x+y\sigma}=0$ for $x<0$.
\end{lemma}
\begin{proof}
As above, by Lemma \ref{lemma_a periodicity of geometric fp and tate} we need only to show the statement for the twisted degree $y=0$. Writing the long exact sequence in homotopy for the upper row of the Tate diagram yields:
\[
\begin{tikzcd}
\ldots\rar & \ho{\hm}{m}\rar & \hm^Q_m\rar & \hm^{\Phi Q}_m\rar &\ho{\hm}{m-1}\rar & \hm^Q_{m-1}\rar &\ldots.
\end{tikzcd}
\]
But $\hm^Q_m=0$ for $m\neq 0$ and $\ho{\hm}{m}=0$ for $m<0$ by Proposition \ref{prop coefficients of hfp and ho}, so $\hm^{\Phi Q}_m=0$ for $m<0$.
\end{proof}
\begin{proof}[Proof of Theorem \ref{thm coefficients of geometric fixed points}]
Follows from Lemmas \ref{lemma gfp1}, \ref{lemma gfp iso to tate} and \ref{lemma gfp 3}.
\end{proof}
\section{$RO(Q)$-graded abelian group structure of $\hm^Q_\star$}
In this section we describe the structure of $\hm^Q_\star$ as a $RO(Q)$-graded abelian group. It is given by the following theorem:
%
%Theorem - RO(Q)-graded abelian group structure of \hm^Q_\star
%
\begin{theorem}
\label{Theorem RO(Q)graded abelian structure of hmq}
The $RO(Q)$-graded abelian group structure of $\hm^Q_\star$ is given by:
\begin{enumerate}
\item
\[
\hm^Q_{y\sigma}=
\begin{cases}
\ker(\resm)&\textrm{if }y>0 \\
\coker(\trm)&\textrm{if }y<0 \\
\M(Q/Q)&\textrm{if }y=0
\end{cases}
\]
\item \[
\hm^Q_{1+y\sigma}=
\begin{cases}
\ker(\trm)&\textrm{if }y=-1 \\
\ker(\trm)_Q&\textrm{if }y<-1 \\
0&\textrm{if }y>-1
\end{cases}
\]
\item 
\[
\hm^Q_{y\sigma-1}=
\begin{cases}
\coker(\resm)&\textrm{if }y=1 \\
V^Q/\im(\resm)&\textrm{if }y>1 \\
0&\textrm{if }y<1
\end{cases}
\]
\item if $x\geq 2$ then $\hm^Q_{x+y\sigma}=\hm^{hQ}_{x+y\sigma}$;
\item if $x\leq -2$ then $\hm^Q_{x+y\sigma}=\ho{\hm}{x+y\sigma}$.
\end{enumerate}
This data is presented in Figure \ref{Figure symbolic depiction}.
\end{theorem}
%
%Figure - RO(Q) graded homotopy groups
%
\begin{figure}[ht]
\begin{tikzpicture}[scale=0.7]
%Kratka

\draw[help lines, thin, lgray] (0,0) grid (20,20);
%\draw[llgray, dashed] (0,20)--(20,0);
\draw[->, thick] (0,10)--(20,10) ;
\draw[->, thick] (10,0)--(10,20) ;
%Homologie
\draw[red] (0,20)--(8,12);
\draw[red] (8,12)--(8,20);
\fill[red!15!white] (0,20)--(8,12)--(8,20);

\node[scale=1.4] at (5.7,17.5) {$\ho{\hm}{x+y\sigma}$};

%Kohomologie
\draw[blue] (12,8)--(20,0);
\draw[blue] (12,8)--(12,0);
\fill[blue!15!white] (20,0)--(12,8)--(12,0);
\node[scale=1.4] at (15,2) {$\hm^{hQ}_{x+y\sigma}$};

\draw[BrickRed, line width=2pt] (9,11)--(9,20);
\node[circle,BrickRed,fill] at (9,11) {};
\node at (7.2,11) {$\coker(\res)$};

\draw[Brown, line width=2pt] (10,10)--(10,19.8);

\node[circle, Periwinkle,fill] at (11,9) {};
\draw[Periwinkle, line width=2pt] (11,0)--(11,9);
\draw[NavyBlue, line width=2pt] (10,0)--(10,10);
\node at (12.4,9) {$\ker(\tr)$};

\node[circle, draw, fill=white, scale=0.5] at (10,10) {$\M(Q/Q)$};

\node[above left] at (20,10) {$x\cdot 1$};
\node[above left, rotate=90] at (10,20) {$y\cdot \sigma$};

\node[scale=1.5] at (10,21.5) {$\hm^Q_{x+y\sigma}$};

\draw [decorate,decoration={brace,amplitude=7pt},xshift=-6pt,yshift=0pt, thick]
(10,0) -- (10,9) node [black,midway,xshift=-1.2cm] 
{$\coker(\tr)$};
\draw [decorate,decoration={brace,amplitude=7pt,mirror},xshift=6pt,yshift=0pt, thick]
(10,11) -- (10,20) node [black,midway,xshift=1cm] 
{$\ker(\res)$};
\draw (7.3,20.2)--(9,18);
\draw (11.8,-0.2)--(11,2);
\node at (12,-0.5) {$\ker(\tr)_Q$};
\node at (8,20.5) {$V^Q/\im(\res)$};
\end{tikzpicture}
\caption{Symbolic depiction of the coefficients of a general Mackey functor.}
\label{Figure symbolic depiction}
\end{figure}
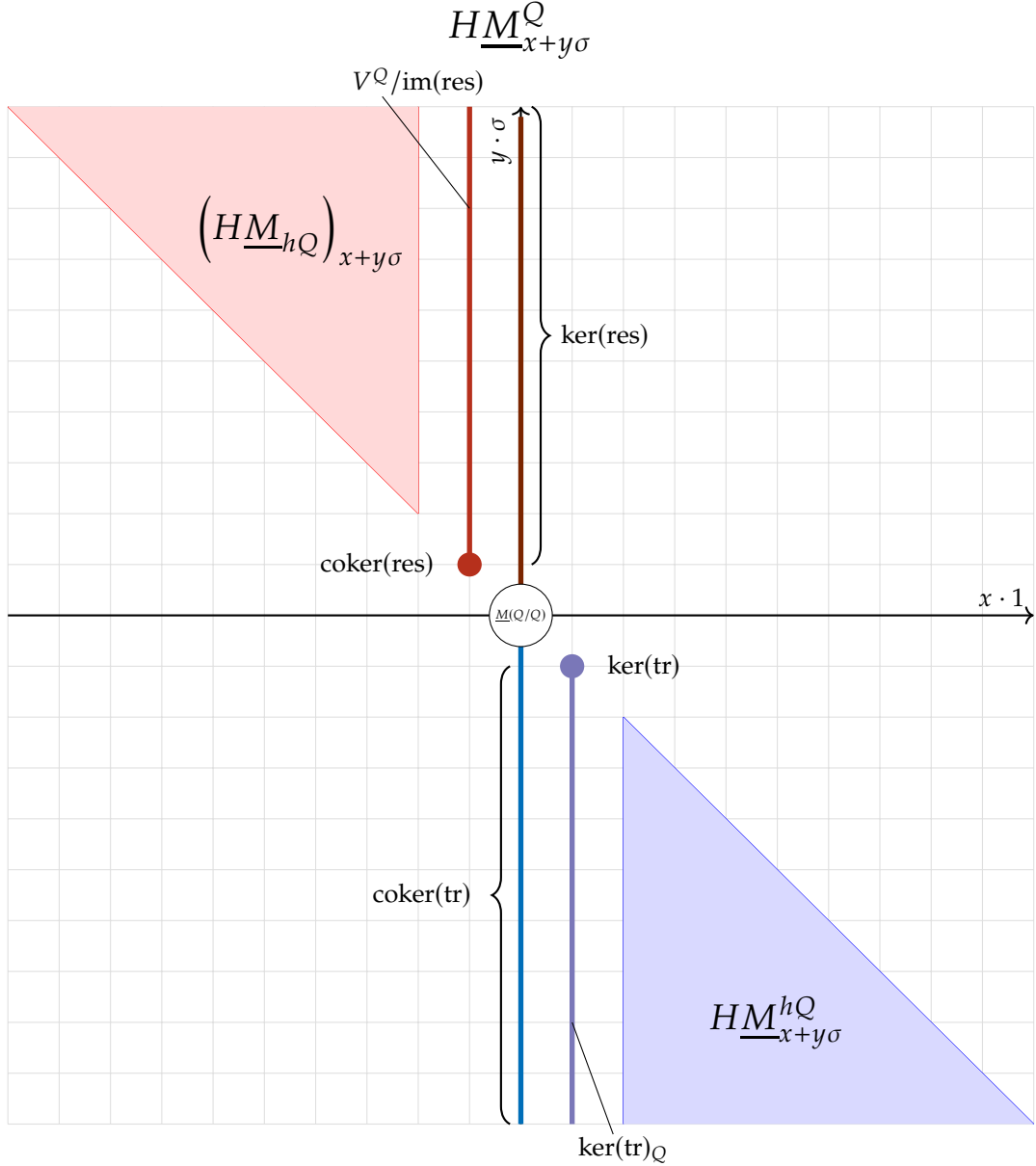
\begin{rem}
Recall from the introduction that the same values can be found in \cite[Theorem 8.1]{MR2025457}. 
\end{rem}
We are going to prove Theorem \ref{Theorem RO(Q)graded abelian structure of hmq} in a series of lemmas. Firstly, let us note that the point (1) of the theorem is actually Proposition \ref{prop groups hmq for x=0}, so it is already proven.

We start with proving a lemma which describes the general shape of $\hm^Q_\star$ - i.e., that it is zero below the antidiagonal on the half-plane $x<0$ and above the antidiagonal on the half-plane $x>0$:
\begin{lemma}
\label{lemma general shape}
For $x<0$, the groups $\hm^Q_{x+y\sigma}$ are zero if $y<-x$. If $x>0$ then $\hm^Q_{x+y\sigma}$ is zero for $y>-x$.
\end{lemma}
\begin{proof}
The groups $\hm^Q_x$ are zero if $x\neq 0$. Let $x>0$. By Lemma \ref{lemma multiplication by a} iterated multiplication by $a$ gives isomorphisms $\hm^Q_x=\hm^Q_{x+y\sigma}$ if $y>-x$. Thus we get that $\hm^Q_{x+y\sigma}=0$ if $y>-x$. The case when $x<0$ follows similarly.
\end{proof}
\begin{lemma}
\label{lemma values 1+ysigma}
$\hm^Q_{1+y\sigma}$ is given by
\[
\hm^Q_{1+y\sigma}=
\begin{cases}
\ker(\trm)&\textrm{if }y=-1 \\
\ker(\trm)_Q&\textrm{if }y<-1 \\
0&\textrm{if }y>-1.
\end{cases}
\]
\end{lemma}
\begin{proof}
Firstly we note that the case $y>-1$ follows from Lemma \ref{lemma general shape}, so it is proven.

Let $y=-1$. If we smash the cofibre sequence \eqref{cofibre} with $\hm$ from Section \ref{sec structure of hm y axis} we obtain the following cofibre sequence:
\[
\begin{tikzcd}
Q_+\wedge\hm\rar& \hm\rar&S^\sigma\wedge\hm.
\end{tikzcd}
\]
This cofibration gives us the following exact sequence in homotopy:
\[
\begin{tikzcd}
\hm^Q_{1}\rar& \hm^Q_{1-\sigma}\rar&\hm^e_0\rar["\trm"]&\hm^Q_0.
\end{tikzcd}
\]
By the definition of Eilenberg-MacLane spectrum we have that $\hm^Q_1=0$. Thus $\hm^Q_{1-\sigma}=\ker(\trm)$.

Finally, let $y<-1$. By applying the long exact sequence in homotopy to the upper row of the Tate diagram we get
\[
\begin{tikzcd}
\ho{\hm}{1+y\sigma}\rar&\hm^Q_{1+y\sigma}\rar&\hm^{\Phi Q}_{1+y\sigma}\rar&\ho{\hm}{y\sigma}.
\end{tikzcd}
\]
By Proposition \ref{prop coefficients of hfp and ho} the outer terms of the sequence above are $0$, so $\hm^Q_{1+y\sigma}\cong\hm^{\Phi Q}_{1+y\sigma}$. By Theorem \ref{thm coefficients of geometric fixed points} we get that $\hm^Q_{1+y\sigma}\cong\ker(\trm)_Q$.
\end{proof}
\begin{lemma}
\label{lemma values sigma-1}
$\hm^Q_{y\sigma-1}$ is given by
\[
\hm^Q_{y\sigma-1}\cong
\begin{cases}
\coker(\resm)&\textrm{if }y=1\\
V^Q/\im(\resm)&\textrm{if } y>1 \\
0&\textrm{if }y<1.
\end{cases}
\]
\end{lemma}
\begin{proof}
The case $y<1$ is proven by Lemma \ref{lemma general shape}. For $y=1$ we use a dual argument to the one given in the proof of Lemma \ref{lemma values 1+ysigma}. Thus we are left with one case.

Let $y>1$. By applying the long exact sequences in homotopy to the Tate diagram we obtain the following commutative diagram:
\[
\label{diagram in the proof}
\begin{tikzcd}
\hm^{\Phi Q}_{y\sigma}\rar\dar["\epsilon_{t\ast}"] &\left(\hm_{hQ}\right)_{y\sigma -1}\rar\dar["="] &\hm^Q_{y\sigma -1}\rar & 0 \\
\hm^{tQ}_{y\sigma}\rar &\left(\hm_{hQ}\right)_{y\sigma -1}&&
\end{tikzcd}
\tag{$\ast\ast$}
\]
Here $\epsilon_{t\ast}$ is the map induced by $\epsilon_t$. The zero in the top right corner comes from the fact that $\hm^{\Phi Q}_{y\sigma-1}=0$ by Theorem \ref{thm coefficients of geometric fixed points}. Note that the bottom arrow is a part of the following sequence:
\[
\begin{tikzcd}
\hm^{hQ}_{y\sigma}\rar&\hm^{tQ}_{y\sigma}\rar &\ho{\hm}{y\sigma-1}\rar&\hm^{hQ}_{y\sigma-1}.
\end{tikzcd}
\]
In this sequence the outer terms are zero by Proposition \ref{prop coefficients of hfp and ho} if $y>1$ and thus the map $\hm^{tQ}_{y\sigma}\to\ho{\hm}{y\sigma-1}$ in the diagram \eqref{diagram in the proof} is an isomorphism. Thus we can deduce from this diagram that $\hm^Q_{y\sigma-1}=\coker(\epsilon_{t\ast})$. So it remains to describe the map $\epsilon_{t\ast}\colon\hm^{\Phi Q}_{y\sigma}\to\hm^{tQ}_{y\sigma}$.

Since both $\hm^{\Phi}$ and $\hm^t$ are $a$-periodic by Lemma \ref{lemma_a periodicity of geometric fp and tate}, it is enough to identify this map for $y=0$. It is described by the following diagram:
\[
\begin{tikzcd}
\hm^Q_0=M (Q/Q) \rar\dar["\widetilde{\resm}"] & \hm^{\Phi Q}_0= \coker (\trm )\dar["\epsilon_{t\ast}"] \\
\hm^{hQ}_0= V^Q\rar & \hm^{tQ}_0=\sfrac{V^Q}{(1+\gamma)V}.
\end{tikzcd}
\]
The map $\widetilde{\resm}\colon \M(Q/Q)\to V^Q$ is the map $\resm$ with the codomain restricted to $V^Q$. Note that by properties of Mackey functors we have that $\im(\resm)\subset V^Q$. Both horizontal maps are canonical projections. So $\epsilon_{t\ast}$ is the map induced by $\resm$. Thus we have:
\[
\hm^Q_{y\sigma-1}\cong \coker(\epsilon_{t\ast})
=
\frac{\left(\sfrac{V^Q}{NV}\right)}{\resm\left(\sfrac{\M (Q/Q)}{\trm (V)}\right)}
\cong
V^Q/\im(\resm).
\]
\end{proof}
\begin{proof}[Proof of Theorem \ref{Theorem RO(Q)graded abelian structure of hmq}]
Follows from Lemmas \ref{lemma general shape}, \ref{lemma values 1+ysigma} and \ref{lemma values sigma-1}.
\end{proof}
\section{Examples}
\label{sec Examples}
In this section we present two examples of the structure shown in previous sections. Since both examples are Green functors, we will also describe a multiplicative structure on the coefficients. The main example is the Burnside Mackey functor $\underline{\mathbb{A}}$, as we are going to build the multiplicative structure of other Green functors upon the knowledge of $\ha^Q_\star$. However, a big part of computations of the coefficients of $\ha$ is the same as in the case of constant Mackey functor $\mathbb{Z}$, coefficients of which are already computed by Tate method in \cite{FourApproaches}. Thus we begin the examples with computations of $\hz^Q_\star$.
\subsection{Constant Mackey functor $\underline{\mathbb{Z}}$}

The Mackey functor $\underline{\mathbb{Z}}$ has the following form:
\[
\mackey{\mathbb{Z}}{\mathbb{Z}}{1}{2}
\]
Computations of $\hz^Q_\star$ as a ring may be found in \cite[Appendix B]{MR2240234} and later on, using the Tate square technique, in \cite{FourApproaches}. From the latter we recall the following two lemmas \cite[Lemma 2.1, Corollary 2.3, Lemma 2.5]{FourApproaches}:
\begin{lemma}
\label{lemma homotopy fixed points hz}
\[
\begin{array}{c}
\hz^{hQ}_\star= BB[u, u^{-1}] \\
\left(\hz_{hQ}\right)_\star= NB[u, u^{-1}] \\
\hz^{tQ}_\star=\mathbb{F}_2[a,a^{-1}][u,u^{-1}]
\end{array}
\]
where $BB=\mathbb{Z}[a]/2a$, $NB=\mathbb{Z}\oplus\Sigma^{-1+\sigma}\mathbb{F}_2[a^{-1}]$ and $|a|=-\sigma$, $|u|=2-2\sigma$.
\end{lemma}

\begin{lemma}
\label{lemma geometric fixed points hz}
The coefficients of $\hz^\Phi$ are given by $\hz^{\Phi Q}_{\star}=\mathbb{F}_2[a,a^{-1}][u]$.
\end{lemma}
\begin{rem}
Note that the element $a$ in Lemmas \ref{lemma homotopy fixed points hz} and \ref{lemma geometric fixed points hz} is the same as the element $a\cdot 1$ in the sense of Lemma \ref{lemma multiplication by a}, so there is no clash of the notation.
\end{rem}
Now we are in a position to describe the structure of $\hz^Q_\star$ as a ring.
%
%Theorem - Coefficients of HZ
%
\begin{theorem}
\label{thm Coefficients of Hz}
\leavevmode
\begin{enumerate}
\item The $RO(Q)$-graded abelian group structure of $\hz^Q_\star$ is given by:
\[
\hz^Q_{x+y\sigma}=
\begin{cases}
\mathbb{Z}&\textrm{if}\;y=-x\;\textrm{and}\;x=2n\;\textrm{for}\;n\in\mathbb{Z} \\
\mathbb{Z}/2&\textrm{if}\;y<-x\;\textrm{and}\;x=2n\;\textrm{for}\;n\in\mathbb{Z}_{\geq 0} \\
\mathbb{Z}/2&\textrm{if}\;y\geq-x\;\textrm{and}\;x=-2n-1\;\textrm{for}\;n\in\mathbb{Z}_{\geq 1} \\
0&\textrm{else.}
\end{cases}
\]
\item The multiplicative structure of $\hz^Q_\star$ is characterized by the following properties:
\begin{enumerate}
\item it is strictly commutative;
\item red lines in Figure \ref{Figure coeffs hz} represent multiplication by $a\in\hz^Q_{-\sigma}$;
\item blue dashed lines represent multiplication by $u\in\hz^Q_{2-2\sigma}$;
\item the map $u\colon\hz^Q_{2\sigma-2}\to\hz^Q_0$ is multiplication by $2$;
\item the groups $\hz^Q_{2n\sigma-2n}$ for $n>0$ are generated by $2u^{-n}$.
\end{enumerate}
\end{enumerate}
\end{theorem}
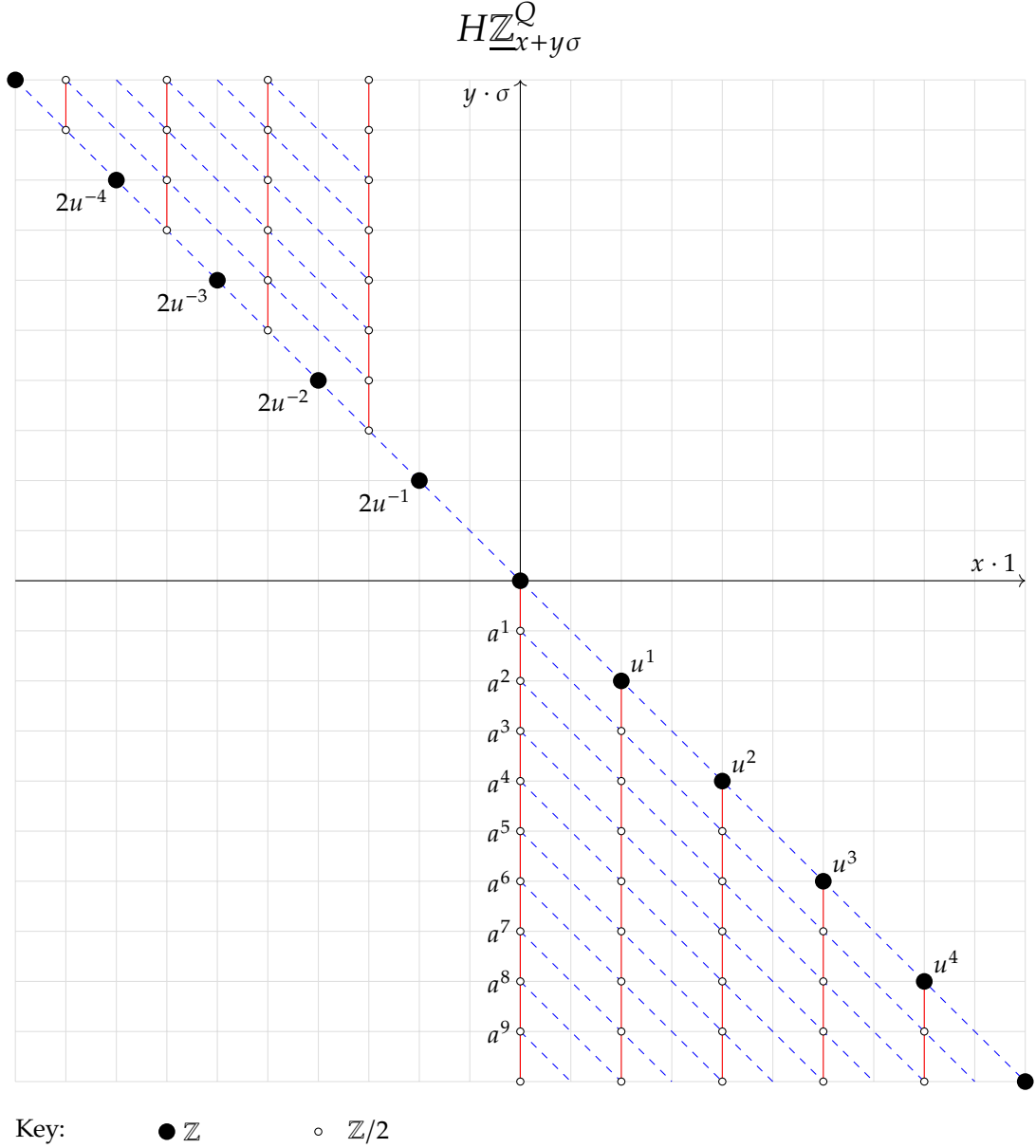
\begin{figure}[ht]
\begin{tikzpicture}[scale=0.7]
%Kratka
\draw[help lines, thin, lgray] (0,0) grid (20,20);
\draw[->] (0,10)--(20,10);
\draw[->] (10,0)--(10,20);
%Czerwone linie po lewej
\draw[red] (10,0)--(10,10);
\draw[red] (1,20)--(1,19);
\draw[red] (3,20)--(3,17);
\draw[red] (5,20)--(5,15);
\draw[red] (7,20)--(7,13);

%CZerwone linie po prawej
\draw[red] (12,8)--(12,0);
\draw[red] (14,0)--(14,6);
\draw[red] (16,0)--(16,4);
\draw[red] (18,0)--(18,2);

%Niebieskie linie
\draw[blue, dashed] (0,20)--(20,0);
\draw[blue, dashed] (1,20)--(7,14);
\draw[blue, dashed] (2,20)--(7,15);
\draw[blue, dashed] (3,20)--(7,16);
\draw[blue, dashed] (4,20)--(7,17);
\draw[blue, dashed] (5,20)--(7,18);

\foreach \x in {1,...,9}
\draw[blue, dashed] (10,\x)--(10+\x,0);

%Grube kropki
\foreach \x in {0,2,4,6,8,10}
\node[circle,fill,scale=0.7] at (10-\x,10+\x) {};
\foreach \x in {2,4,6,8,10}
\node[circle,fill,scale=0.7] at (10+\x,10-\x) {};

%Etykiety
\foreach \x in {1,...,4}
\pgfmathsetmacro\r{\x*2}
\node[below left] at(10-\r,10+\r) {$2u^{-\x}$};
\foreach \x in {1,...,4}
\pgfmathsetmacro\r{\x*2}
\node[above right] at (10+\r,10-\r) {$u^{\x}$};
\foreach \x in {1,...,9}
\node[left] at (10, 10-\x) {$a^\x$};

%Pustekropki w górę
\foreach \x in {0,...,1}
\node[circle, draw, fill=white, scale=0.3] at (1,20-\x) {};
\foreach \x in {0,...,3}
\node[circle, draw, fill=white, scale=0.3] at (3,20-\x) {};
\foreach \x in {0,...,5}
\node[circle, draw, fill=white, scale=0.3] at (5,20-\x) {};
\foreach \x in {0,...,7}
\node[circle, draw, fill=white, scale=0.3] at (7,20-\x) {};

%Pustekropki w dół
\foreach \x in {0,...,1}
\node[circle, draw, fill=white, scale=0.3] at (18,\x) {};
\foreach \x in {0,...,3}
\node[circle, draw, fill=white, scale=0.3] at (16,\x) {};
\foreach \x in {0,...,5}
\node[circle, draw, fill=white, scale=0.3] at (14,\x) {};
\foreach \x in {0,...,7}
\node[circle, draw, fill=white, scale=0.3] at (12,\x) {};
\foreach \x in {0,...,9}
\node[circle, draw, fill=white, scale=0.3] at (10,\x) {};

\node[above left] at (20,10) {$x\cdot 1$};
\node[below left] at (10,20) {$y\cdot \sigma$};

\node[scale=1.5] at (10,21) {$H\underline{\mathbb{Z}}^Q_{x+y\sigma}$};

\node at (0.5,-1) {Key:};
\node[circle, fill,scale=0.7] at (3,-1) {};
\node at (3.5,-1) {$\mathbb{Z}$};
\draw (6,-1) circle (2pt);
\node at (7,-1) {$\mathbb{Z}/2$};
\end{tikzpicture}
\caption{Coefficients of $\hz$.}
\label{Figure coeffs hz}
\end{figure}
\begin{proof}
\leavevmode
\begin{enumerate}
\item This comes from Theorem \ref{Theorem RO(Q)graded abelian structure of hmq}:
\begin{itemize}
\item if $x\leq -2$ we have that $\hz^{Q}_{x+y\sigma}=\ho{\hz}{x+y\sigma}$, so the structure is described by Lemma \ref{lemma homotopy fixed points hz};
\item if $x\geq 2$ then $\hz^Q_{x+y\sigma}=\hz^{hQ}_{x+y\sigma}$, so the structure is also described by Lemma \ref{lemma homotopy fixed points hz};
\item if $-1\leq x\leq 1$ the structure comes from the following calculations:
\begin{align*}
\hz^Q_{1-\sigma}&=\ker(\tr_{\underline{\mathbb{Z}}})=0\\
\hz^Q_{1-y\sigma}&=\ker(\tr_{\underline{\mathbb{Z}}})_Q=0 \textrm{ for } y>1\\
\hz^Q_{\sigma-1}&=\coker(\res_{\underline{\mathbb{Z}}})=0\\
\hz^Q_{y\sigma-1}&=V^Q/\im(\res_{\underline{\mathbb{Z}}})=0\textrm{ for } y>1\\
\hz^Q_{y\sigma}&=\ker(\res_{\underline{\mathbb{Z}}})=0\textrm{ for } y>0\\
\hz^Q_{y\sigma}&=\coker(\tr_{\underline{\mathbb{Z}}})=\mathbb{Z}/2\textrm{ for } y<0
\end{align*}
\end{itemize}
\item
\begin{enumerate}
\item Since $\underline{\mathbb{Z}}$ is a Green functor, $\hz$ is a commutative ring $Q$-spectrum and so $\hz^Q_\star$ is a graded commutative ring. The graded commutativity rule is given as follows: if $\alpha\in \hz^Q_{x+y\sigma}$ and $\beta\in \hz^Q_{x'+y'\sigma}$ then 
\[
\alpha\beta=(-1)^{xx'}(1-\omega)^{yy'}\beta\alpha,
\]
where $\omega$ is the class of $Q/e$ in $\aq$ (see Observation \ref{obs RO(Q)-graded ring} and Remark \ref{rem graded comm rule}). However, since all entries except of the antidiagonal are either zero or $\mathbb{Z}/2$, the sign rule might give non-trivial sign only in $\mathbb{Z}$'s on the antidiagonal. But they all lie in even fixed and twisted degrees, so the sign is always 1. 
\item The multiplication by $a$ is described by Lemma \ref{lemma multiplication by a}.
\item Firstly recall from Section \ref{sec The Tate method} that the map $\epsilon_\ast\colon\hz^Q_\star\to\hz^{hQ}_\star$ is a ring map. Thus for $x\geq 0$ multiplication by the element $u$ is described by Lemma \ref{lemma homotopy fixed points hz}. By Proposition \ref{prop X_h is a module over X^h} we have that $\ho{\hz}{\star}$ is a module over $\hz^Q_\star$, so the claim follows from Lemma \ref{lemma homotopy fixed points hz}.
\item We have the following commutative diagram:
\[
\begin{tikzcd}
\left( \hz_{hQ}\right)_{2\sigma-2} \rar["u"]\dar &   \left( \hz_{hQ}\right)_0 \dar["f_0"] \\
\hz^Q_{2\sigma-2}\rar["u"] & \hz^Q_0.
\end{tikzcd}
\]
By Lemma \ref{lemma homotopy fixed points hz} the upper horizontal arrow is an isomorphism. By Theorem \ref{Theorem RO(Q)graded abelian structure of hmq} the left vertical arrow is an isomorphism, so multiplication  $u\colon\hz^Q_{2\sigma-2}\to\hz^Q_0$ is up to isomorphism the same as the right vertical map $f_0$. By Proposition \ref{prop epsilon_0 f_0} this map is induced by the transfer. Thus if $\alpha\in\hz^Q_{2\sigma-2}$ we obtain that $u\cdot\alpha=2\alpha\in\hz^Q_0$.
\item Let $\theta$ be the generator corresponding to $1$ in $\hz^Q_{2\sigma-2}\cong\mathbb{Z}$. By the previous point $u\cdot\theta=2$, so $\theta=2u^{-1}$. Since $\hz^Q_{x+y\sigma}\cong\ho{\hz}{x+y\sigma}$ if $x\leq -2$ and multiplication by $u$ is an isomorphism on $\ho{\hz}{\star}$, we get that $\hz^Q_{2n\sigma-2n}$ is generated by $2u^{-n}$ for $n>0$.
\end{enumerate}
\end{enumerate}
\end{proof}
\begin{rem}
The ring $\hz^Q_{\star}$ has a more concise description as:
\[
\hz^Q_{\star}\cong BB[u]\oplus u^{-1}\cdot NB[u^{-1}]
\]
with $BB$ as before and $NB=2\mathbb{Z}\oplus\Sigma^{-1+\sigma}\mathbb{F}_2[a^{-1}]$. The structure of $BB$-module of $NB$ is as suggested by the notation. See \cite[Corollary 2.6]{FourApproaches}.
\end{rem}
\subsection{Burnside Mackey functor}
The second example is the Burnside Mackey functor $\A$. It is the most important example - since we will build our understanding of multiplicative structure of Eilenberg-MacLane spectra upon the knowledge of the graded ring structure of $\ha^Q_\star$. 

The Mackey functor $\A$ has the following form:
\[
\mackey{\sfrac{\mathbb{Z}[\omega]}{\omega^2-2\omega}}{\mathbb{Z}}{}{}
\]
with the transfer given by multiplication by $\omega$ and the restriction given by evaluating at $\omega=2$. Note that the $Q/Q$-level of $\A$ is in fact the Burnside ring $\mathbb{A}(Q)$ with $\omega$ being the class of $Q/e$. The Mackey functor valued coefficients of $\A$ may be found in \cite[Section 2]{MR979507} and \cite[Proposition 1.7(b)]{MR2025457}.

The coefficients of $\ha^h$, $\ha_h$ and $\ha^t$ depend only on $\A(Q/e)=\mathbb{Z}$ (see Propositions \ref{prop coefficients of hfp and ho} and \ref{prop coefficients of Tate}), which is the same as $\underline{\mathbb{Z}}(Q/e)$. So $\ha^{hQ}_\star\cong\hz^{hQ}_\star$, analogously for $\ho{\hz}{\star}$ and $\ha^{tQ}_\star$. Thus Lemma \ref{lemma homotopy fixed points hz} gives a description for this entries of the Tate diagram.

Now we need to compute the coefficients of $\ha^\Phi$:
\begin{lemma}
\[
\ha^{\Phi Q}_\star\cong\mathbb{Z}[a,a^{-1}][u]/2u.
\]
\end{lemma}
\begin{proof}
By Theorem \ref{thm coefficients of geometric fixed points} we have that $\ha^{\Phi Q}_0\cong\mathbb{Z}$ and $\ha^{\Phi Q}_1\cong 0$. So by $a$-periodicity of $\ha^\Phi$ (see Lemma \ref{lemma_a periodicity of geometric fp and tate}) we obtain that $\ha^{\Phi Q}_{\ast\sigma}\cong\mathbb{Z}[a,a^{-1}]$. Since the map $\epsilon_{t\star}\colon\ha^{\Phi Q}_{x+y\sigma}\to\ha^{tQ}_{x+y\sigma}$ induced by $\epsilon_t$ is a ring map and it is an isomorphism if $x\geq 2$, the result follows by Lemma \ref{lemma homotopy fixed points hz}.
\end{proof}

\begin{theorem}
\label{thm coeffs of ha}
\leavevmode
\begin{enumerate}
\item The $RO(Q)$-graded abelian group structure of $\ha^Q_\star$ is given by:
\[
\ha^Q_{x+y\sigma}\cong
\begin{cases}
\mathbb{A}(Q)&\textrm{if }x=y=0 \\
\mathbb{Z}&\textrm{if }x=0\textrm{ and }y\neq 0\\
\mathbb{Z}&\textrm{if }x\textrm{ even and }y=-x\\
\mathbb{Z}/2&\textrm{if }x\textrm{ odd, }x\leq -3\textrm{ and }y>x\\
\mathbb{Z}/2&\textrm{if }x\textrm{ even, }x\geq 2\textrm{ and }y<x\\
0&\textrm{else }.
\end{cases}
\]
\item The multiplicative structure of $\ha^Q_\star$ is given by the following properties:
\begin{enumerate}
\item it is strictly commutative;
\item red lines on Figure \ref{figure coeffs of ha} represent multiplication by $a$;
\item blue dashed lines represent multiplication by $u$, the generator of $\ha^Q_{2-s\sigma}$ corresponding to $1$;
\item if $\tau$ is a generator of $\ha^Q_{\sigma}$, then $a\tau=\omega-2$;
\item $u\colon\ha^Q_{2\sigma-2}\to\ha^Q_0$ is the transfer map in $\A$ and $u\colon\ha^Q_{0}\to\ha^Q_{2-2\sigma}$ is the restriction;
\item for $n>0$ the group $\ha^Q_{2n\sigma-2n}$ is generated by $\omega u^{-n}$.
\end{enumerate}
\end{enumerate}
In particular, the subring consisting of entries for $x\geq 0$ and $y\leq 0$ is a truncated polynomial algebra
\[\frac{\mathbb{A}(Q)[a,u]}{a\omega,2au}.\] 

The data above is presented in Figure \ref{figure coeffs of ha}.
\end{theorem}
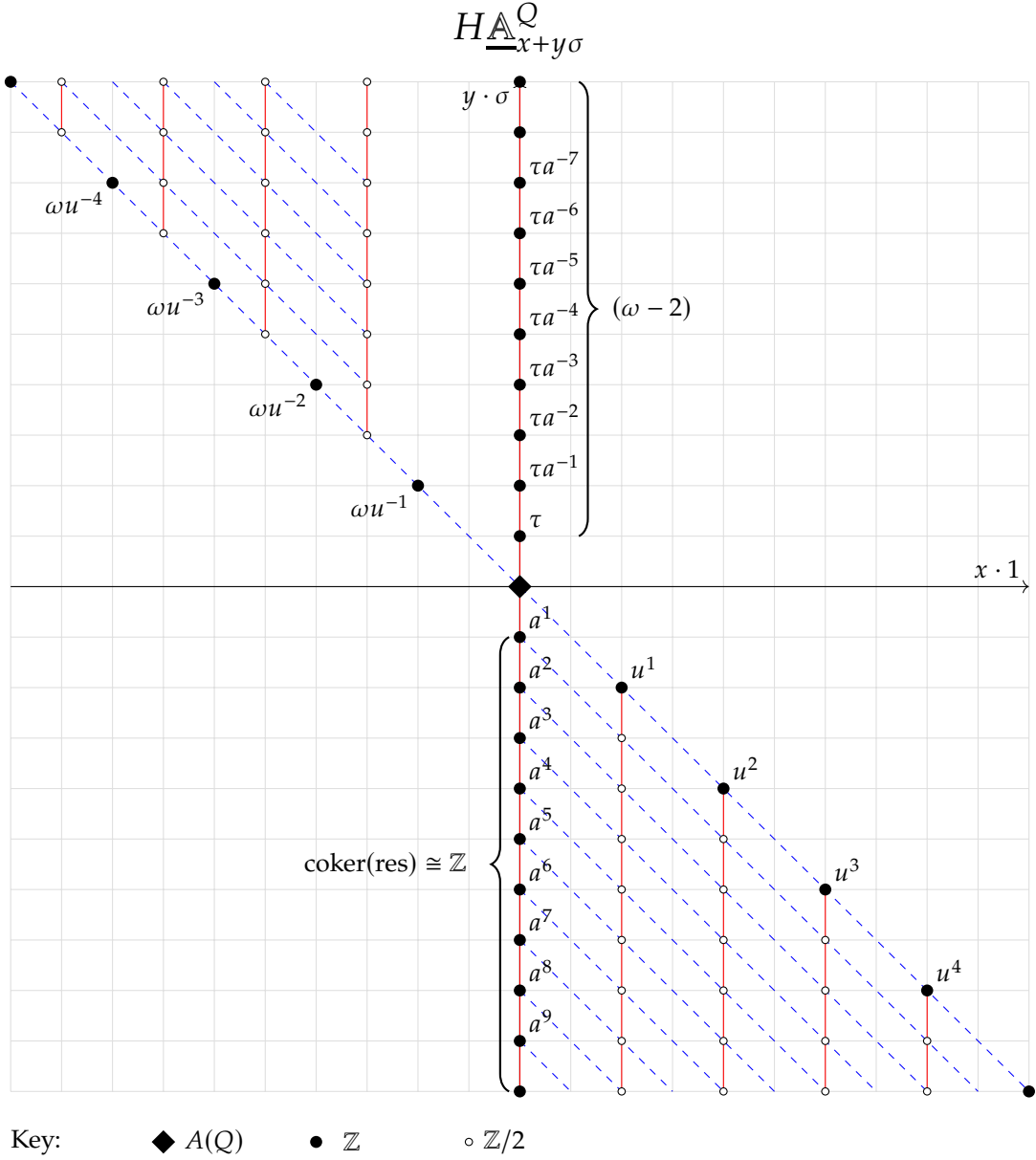
\begin{figure}[ht]
\begin{tikzpicture}[scale=0.7]
%Kratka
\draw[help lines, thin, lgray] (0,0) grid (20,20);
\draw[->] (0,10)--(20,10) ;
\draw[->] (10,0)--(10,20) ;

%Czerwone linie po lewej
\draw[red] (10,0)--(10,20);
\draw[red] (1,20)--(1,19);
\draw[red] (3,20)--(3,17);
\draw[red] (5,20)--(5,15);
\draw[red] (7,20)--(7,13);

%CZerwone linie po prawej
\draw[red] (12,8)--(12,0);
\draw[red] (14,0)--(14,6);
\draw[red] (16,0)--(16,4);
\draw[red] (18,0)--(18,2);

%Niebieskie linie
\draw[blue,dashed] (0,20)--(20,0);
\draw[blue, dashed] (1,20)--(7,14);
\draw[blue, dashed] (2,20)--(7,15);
\draw[blue, dashed] (3,20)--(7,16);
\draw[blue, dashed] (4,20)--(7,17);
\draw[blue,dashed] (5,20)--(7,18);

\foreach \x in {1,...,9}
\draw[blue,dashed] (10,\x)--(10+\x,0);

%Grube kropki
\foreach \x in {2,4,6,8,10}
\node[circle,fill,scale=0.5] at (10-\x,10+\x) {};
\foreach \x in {2,4,6,8}
\pgfmathsetmacro\result{\x/2}
\node[below left] at (10-\x,10+\x) {$\omega u^{-\pgfmathprintnumber{\result}}$};
\foreach \x in {2,4,6,8,10}
\node[circle,fill,scale=0.5] at (10+\x,10-\x) {};
\foreach \x in {2,4,6,8}
\pgfmathsetmacro\result{\x/2}
\node[above right] at (10+\x,10-\x) {$u^{\pgfmathprintnumber{\result}}$};
\foreach \x in {0,...,9}
\node[circle,fill,scale=0.5] at (10,\x) {};
\foreach \x in {1,...,9}
\node[above right] at (10,10-\x) {$a^\x$};
\foreach \x in {0,...,9}
\node[circle,fill,scale=0.5] at (10,20-\x) {};
\foreach \x in {1,...,7}
\node[above right] at (10,11+\x) {$\tau a^{-\x}$};
\node[above right] at (10, 11) {$\tau$};

%Pustekropki w górę
\foreach \x in {0,...,1}
\node[circle, draw, fill=white, scale=0.3] at (1,20-\x) {};
\foreach \x in {0,...,3}
\node[circle, draw, fill=white, scale=0.3] at (3,20-\x) {};
\foreach \x in {0,...,5}
\node[circle, draw, fill=white, scale=0.3] at (5,20-\x) {};
\foreach \x in {0,...,7}
\node[circle, draw, fill=white, scale=0.3] at (7,20-\x) {};

%Pustekropki w dół
\foreach \x in {0,...,1}
\node[circle, draw, fill=white, scale=0.3] at (18,\x) {};
\foreach \x in {0,...,3}
\node[circle, draw, fill=white, scale=0.3] at (16,\x) {};
\foreach \x in {0,...,5}
\node[circle, draw, fill=white, scale=0.3] at (14,\x) {};
\foreach \x in {0,...,7}
\node[circle, draw, fill=white, scale=0.3] at (12,\x) {};

%Diamencik
\node[diamond, fill,scale=0.7] at (10,10) {};

\node[above left] at (20,10) {$x\cdot 1$};
\node[below left] at (10,20) {$y\cdot \sigma$};

\node[scale=1.5] at (10,21) {$\ha^Q_{x+y\sigma}$};

\draw [decorate,decoration={brace,amplitude=7pt},xshift=-6pt,yshift=0pt, thick]
(10,0) -- (10,9) node [black,midway,xshift=-1.7cm] 
{$\coker(\res)\cong\mathbb{Z}$};
\draw [decorate,decoration={brace,amplitude=7pt,mirror},xshift=33pt,yshift=0pt, thick]
(10,11) -- (10,20) node [black,midway,xshift=1cm] 
{$(\omega-2)$};

\node at (0.5,-1) {Key:};
\node[diamond, fill,scale=0.7] at (3,-1) {};
\node at (4,-1) {$A(Q)$};
\node[circle, fill,scale=0.5] at (6,-1) {};
\node at (6.7,-1) {$\mathbb{Z}$};
\node[circle, draw, fill=white, scale=0.3] at (9,-1) {};
\node at (9.7,-1) {$\mathbb{Z}/2$};
\end{tikzpicture}
\caption{Coefficients of $\ha$}
\label{figure coeffs of ha}
\end{figure} 
\begin{proof}
We provide here only the proof of points 1 and 2a, since the rest is completely analogous to the case of $\hz$ given in Theorem \ref{thm Coefficients of Hz}.
\begin{enumerate}
\item If $-1\leq x\leq 1$ then the statement comes from the following calculations:
\[
\begin{array}{l}
\ker(\res_{\A})=(\omega-2)\cong \mathbb{Z}\textrm{ as an abelian group}\\
\coker(\res_{\A})=0 \\
\ker(\tr_{\A})=0 \\
\coker(\tr_{\A})=\mathbb{Z}.
\end{array}
\]
The rest follows from Lemma \ref{lemma homotopy fixed points hz} in the same way as in the proof of Theorem \ref{thm Coefficients of Hz}.
\item
\begin{enumerate}
\item Firstly note that $\A$ is a Green functor, so $\ha$ is a commutative ring $Q$-spectrum. Recall the graded commutativity rule from the proof of Theorem \ref{thm Coefficients of Hz}, part 2a: if $\alpha\in \ha^Q_{x+y\sigma}$ and $\beta\in \ha^Q_{x'+y'\sigma}$ then 
\[
\alpha\beta=(-1)^{xx'}(1-\omega)^{yy'}\beta\alpha.
\]

Since all non-zero entries have even fixed degree, the first unit is always $1$. So we need to show that $1-\omega$ also acts as $1$ in all cases. To this end we need to show that this claim holds only on the antidiagonal and on the $y$-axis, as all other non-zero entries are $\mathbb{Z}/2$.

On the antidiagonal all non-zero entries are in even twisted degrees and so $1-\omega$ acts as $1$. For $y>0$ we have that $\ha^Q_{y\sigma}=\ker(\res_{\A})=(\omega-2)$. Multiplication by $\omega$ on this ideal gives zero, since $\omega(\omega-2)=\omega^2-2\omega=0$ in $\mathbb{A}(Q)$. So $1-\omega$ acts as $1$ on $\ha^Q_{y\sigma}$ if $y>0$. Finally, let $y<0$. In this case we have that $\ha^Q_{y\sigma}=\coker(\tr_{\A})=\mathbb{A}(Q)/\omega$, so $\omega$ acts as $0$ and $1-\omega$ as $1$.
\end{enumerate}
\end{enumerate}
\end{proof}
\section{Periodicity in the antidiagonal direction}
\label{sec multiplication by u}
Sections \ref{sec The Tate method} and \ref{sec Examples} suggest that there exist some patterns in the coefficients of Eilenberg-MacLane $Q$-spectra. One of them is a repetition along the vertical lines, which by Lemma \ref{lemma multiplication by a} we can attribute to multiplication by $a$. The other one may be seen in the antidiagonal direction; we are going to describe it in this section.

Let $u$ be the generator of $\ha^Q_{2-2\sigma}=\ha^{hQ}_{2-2\sigma}\cong\mathbb{Z}$.
%
%Theorem-multiplication by u
%
\begin{theorem}
\label{thm multiplication by u}
For any Mackey functor $\M$ the map $u\colon\hm^Q_{x+y\sigma}\to\hm^Q_{(x+2)+(y-2)\sigma}$ is:
\begin{enumerate}
\item the map induced by the transfer $V_Q\to\M(Q/Q)$ if $x=-2$ and $y=2$;
\item the map induced by the transfer $\prescript{}{N}V/(1-\gamma)V\to\ker(\resm)$ if $x=-2$ and $y>2$;
\item the restriction map $\M(Q/Q)\to V^Q$ if $x=y=0$;
\item the map induced by the restriction $\coker(\trm)\to V^Q/NV$ if $x=0$ and $y<0$;
\item the inclusion $\ker(\trm)\to \prescript{}{N}V$ if $x=1$ and $y=-1$;
\item the map $\ker(\trm)_Q\to\prescript{}{N}V/(1-\gamma)V$ induced by the inclusion from Point 4 if $x=1$ and $y<-1$;
\item multiplication by $1-\gamma$ if $x=-1$ and $y=1$;
\item the zero map if $x=-1$ and $y>1$;
\item the projection $V/NV\to\coker(\resm)$ if $x=-3$ and $y=3$;
\item the projection $V^Q/NV\to V^Q/\im(\resm)$ induced by projection from Point 9 if $x=-3$ and $y>3$;
\item the zero map if $x<0$ and $y<-x$ or $x\geq 0$ and $y>-x$;
\item an isomorphism otherwise.
\end{enumerate}
\end{theorem}
Before proving this theorem we give a couple of preparatory lemmas. Note that since $\hm$ is a module over $\ha$ we have that $\hm^h$ is a module over $\ha^h$ (see \cite[Proposition 3.5]{MR1230773}). The equivalence \[\epsilon\colon EQ_+\wedge X\to EQ_+\wedge F(EQ_+,X)\] gives $\hm_h$ an $\ha^h$-module structure in the homotopy category by the following composite:
\[
\begin{tikzcd}[column sep=large]
EQ_+\wedge F(EQ_+,\ha)\wedge\hm\rar["\bar{\epsilon}\wedge\hm"]&
EQ_+\wedge\ha\wedge\hm\rar&
EQ_+\wedge\hm. 
\end{tikzcd}
\]
Here $\bar{\epsilon}$ is the inverse of $\epsilon$ in the homotopy category. The domain of the composite is $\ha^h\wedge\hm_h$ after applying the appropriate twist map.
\begin{lemma}
\label{lemma ho hfp are uperiodic}
For every Mackey functor $\M$ the modules $\hm^{hQ}_\star$ and $\ho{\hm}{\star}$ are $u$-periodic, i.e., $u$ acts as a unit. 
\end{lemma}
\begin{proof}
By Lemma \ref{lemma homotopy fixed points hz} this is true for $\ha$. 

By \cite[Proposition 8.4]{MR1230773} the pairing
\[\begin{array}{c}
\ha^h\wedge\hm^h\to\hm^h
\end{array}
\] gives the pairing in the group cohomology:
\[
H^*(Q;\mathbb{Z})\otimes H^*(Q;V)\to H^*(Q;V)
\]
Since $\ha^{hQ}_{2-2\sigma}=H^0(Q;\mathbb{Z})\cong\mathbb{Z}$ and $u$ is a generator, the statement follows for $\hm^{hQ}_\star$. By the discussion preceding the lemma $\ho{\hm}{\star}$ is a module over $\ha^{hQ}_\star$, so the second part follows analogously.  
\end{proof}

Note that if $X$ is a $Q$-spectrum, its homotopy groups $X^Q_\star$ and $X^e_\star$ form an $RO(Q)$-graded Mackey functor denoted by $X^\bullet_\star$. We investigate here the Mackey functor structure of two entries of $\hm^\bullet_\star$, namely $\hm^\bullet_{1-\sigma}$ and $\hm^\bullet_{\sigma-1}$. Recall the notation $\tilde{V}=\tilde{\mathbb{Z}}\otimes V$.
\begin{lemma}
\label{lemma hm^bullet_sigma-1}
The Mackey functor structure of $\hm^\bullet_{\sigma-1}$ is given by:
\[
\mackey{\coker(\resm)}{\tilde{V},}{N}{\pi}
\]
where $\pi$ is the map induced by projection of $V$ onto $\coker(\resm)$.
\end{lemma}
\begin{proof}
By Proposition \ref{prop isomorphism from GM} we have that:
\begin{align*}
\hm^e_{\sigma-1}&=[Q_+\wedge S^{\sigma-1},\hm]^Q\\
&\cong [Q_+\wedge S^{-1}, F(S^\sigma,\hm)]^Q\\
&\cong\Hom_{\mathbb{Z}[Q]}\left(H_{-1}(Q_+\wedge S^{-1}),\pi_{-1}(F(S^\sigma,\hm))\right)\\
&\cong H^1(S^\sigma, V)\cong\tilde{V}.
\end{align*} 
The last isomorphism comes from Remark \ref{rem cohomology of S^nsigma}.

From Theorem \ref{Theorem RO(Q)graded abelian structure of hmq} we get that $\hm^Q_{\sigma-1}=\coker(\resm)$.

Recall the cofibre sequence \eqref{cofibre} from Section \ref{sec structure of hm y axis}:
\[
\begin{tikzcd}
Q_+\rar& S^0\rar& S^\sigma.
\end{tikzcd}
\]
After smashing this sequence with $\hm$ and applying $[S^{\sigma-1},-]^Q$ we obtain the following exact sequence of abelian groups, where $\tr$ denotes the transfer in $\hm^\bullet_{\sigma-1}$:
\[
\begin{tikzcd}
\hm^e_{\sigma}\rar&\hm^Q_\sigma\rar&\hm^Q_0\rar&\hm^e_{\sigma-1}\rar["\tr"] & \hm^Q_{\sigma-1} \rar & \hm^Q_{-1}.
\end{tikzcd}
\]
The outer terms are zero, so from this exact sequence and Theorem \ref{Theorem RO(Q)graded abelian structure of hmq} we can read that the underlying map of abelian groups of $\tr$ is the projection of $V$ onto $\coker(\resm)$.

Let $\tilde{x}\in \tilde{V}$ and $x\in V$ be its underlying element in $V$. Then the transfer of $\hm^\bullet_{\sigma-1}$ is given by 
\[
\tilde{x}\mapsto x+\im(\resm).
\]
Note that this map satisfies the condition $\tr(\tilde{x})=\tr(\gamma \tilde{x})$ for transfer in a Mackey functor.

Finally, by the definition of a Mackey functor we have that $\res(\tr(\tilde{x}))=N\tilde{x}$. Thus the restriction in $\hm^Q_{\sigma-1}$ has to be the map
\[
x+\im(\resm)\mapsto N\tilde{x}.
\]
An easy calculation shows that this map is well-defined and satisfies the required properties.
\end{proof}
\begin{lemma}
\label{lemma hm^bullet_1-sigma}
The Mackey functor structure of $\hm^\bullet_{1-\sigma}$ is given by:
\[
\mackey{\ker(\trm)}{\tilde{V},}{i}{N}
\]
where $i$ is the inclusion of $\ker(\trm)$ in $(\tilde{V})^Q=\prescript{}{N}V$.
\end{lemma}
\begin{proof}
This is analogous to the proof of Lemma \ref{lemma hm^bullet_sigma-1}. Note that the identification $(\tilde{V})^Q=\prescript{}{N}V$ follows from the following calculation: $\tilde{x}\in (\tilde{V})^Q$ if and only if $\tilde{x}-\gamma\tilde{x}=0$. But the last equality on the underlying element in $V$ gives that $x+\gamma x=0$, i.e., $x\in\prescript{}{N}V$.
\end{proof}

The following generalisation of Proposition \ref{prop epsilon_0 f_0} will be also of use here:
\begin{prop}
\label{prop f_V epsilon_V}
The map $f_V\colon\left(\hm_{hQ}\right)_V\to\hm^Q_V$ is the map induced by the transfer in the Mackey functor $\hm^\bullet_V$. The map $\epsilon_V\colon \hm^Q_V\to\hm^{hQ}_V$ is the map induced by the restriction in the Mackey functor $\hm^\bullet_V$.
\end{prop}
\begin{proof}
Follows analogously to the proof of Proposition \ref{prop epsilon_0 f_0}.
\end{proof}
%
%Proof of multiplication by u
%
\begin{proof}[Proof of Theorem \ref{thm multiplication by u}]
\leavevmode
\begin{enumerate}
%Punkt1
\item
Multiplication by $u$ on $\ho{\hm}{\star}$ and $\hm^Q_\star$ gives us the following commutative diagram:
\[
\begin{tikzcd}
\left(\hm_{hQ}\right)_{y\sigma-2} \rar["u"]\dar["f_{y\sigma-2}"] & \left(\hm_{hQ}\right)_{(y-2)\sigma}\dar["f_{(y-2)\sigma}"] \\
\hm^Q_{y\sigma-2}\rar["u"] & \hm^Q_{(y-2)\sigma}.
\end{tikzcd}
\]
The left vertical arrow is an isomorphism by Theorem \ref{Theorem RO(Q)graded abelian structure of hmq} and the top arrow is an isomorphism by Lemma \ref{lemma ho hfp are uperiodic}. Thus the right vertical arrow is up to isomorphism the same as the bottom arrow.
 
If $y=2$, the right vertical arrow is the map $f_0$. Thus by Proposition \ref{prop epsilon_0 f_0} the bottom arrow is the map induced by transfer on $V_Q$.
%Pkt2
\item We will use the same commutative diagram as in Point 1. If $y>2$, the right vertical arrow is the map $f_{(y-2)\sigma}$. By Corollary \ref{cor fysigma for y>1}, this is the map $\prescript{}{N}V/(1-\gamma)V\to\ker(\resm)$ induced by $\trm$. Similarly as in Point 1, the right vertical arrow is the same as the bottom arrow, which is multiplication by $u$.
%Punkt 3
\item Follows analogously to Point 1 by considering the map $\epsilon\colon\hm\to\hm^h$.
%Punkt 4
\item By Lemma \ref{lemma multiplication by a} we need to prove the statement only for $y=-1$. By Theorem \ref{thm coeffs of ha} the ring $\ha^Q_\star$ is strictly commutative, so the following diagram is commutative:
\[
\begin{tikzcd}
\hm^Q_0\cong \M(Q/Q)\rar["u"]\dar["a"]&\hm^Q_{2-2\sigma}\cong V^Q\dar["a"]\\
\hm^Q_{-\sigma}\cong\coker(\trm)\rar["u"]&\hm^Q_{2-3\sigma}\cong V^Q/NV.
\end{tikzcd}
\]
By Lemma \ref{lemma multiplication by a} the vertical arrows in this diagram are projections. From the previous point we get that the top arrow is the restriction map $\resm$, thus the bottom arrow is the map induced by the restriction $\coker(\trm)\to V^Q/NV$.
%Pkt5
\item Note that three corners of the diagram displaying multiplication by $u$ are isomorphic:
\[
\begin{tikzcd}
\hm^Q_{1-\sigma}\dar["\epsilon_{1-\sigma}"]\rar["u"] &\hm^Q_{3-3\sigma}\dar["\epsilon_{3-3\sigma}"] \\
\hm^{hQ}_{1-\sigma}\rar["u","\cong"'] & \hm^{hQ}_{3-3\sigma}. 
\end{tikzcd}
\]
The right vertical arrow is an isomorphism by Theorem \ref{Theorem RO(Q)graded abelian structure of hmq} and the bottom arrow is an isomorphism by Lemma \ref{lemma ho hfp are uperiodic}. Thus the top arrow is the same as the left vertical arrow, which by Proposition \ref{prop f_V epsilon_V} is induced by the restriction of the Mackey functor $\hm^\bullet_{1-\sigma}$. So it is the inclusion $\ker(\trm)\to (\tilde{V})^Q=\prescript{}{N}V$.

%Punkt6
\item By the same argument as in Point 5 we have that $u\colon\hm^Q_{1+y\sigma}\to\hm^{Q}_{3+(y-2)\sigma}$ can be identified with $\epsilon_{1+y\sigma}\colon \hm^Q_{1+y\sigma}\to\hm^{hQ}_{1+y\sigma}$. To identify this map, we use the following diagram:
\[
\label{*}
\begin{tikzcd}
\hm^Q_{1+y\sigma}\dar["\epsilon_{1+y\sigma}"] \rar &\hm^{\Phi Q}_{1+y\sigma}\dar["\zeta"] \\
\hm^{hQ}_{1+y\sigma}\rar & \hm^{tQ}_{1+y\sigma}.
\end{tikzcd}
\tag{$\ast$}
\]
If we apply the long exact sequence in homotopy to the top and bottom rows of the Tate diagram we obtain:
\[
\begin{tikzcd}
\ldots\rar&\ho{\hm}{1+y\sigma}\rar&\hm^Q_{1+y\sigma}\rar&\hm^{\Phi Q}_{1+y\sigma}\rar& \ho{\hm}{y\sigma}\rar&\ldots
\end{tikzcd}
\]
and
\[
\begin{tikzcd}
\ldots\rar&\ho{\hm}{1+y\sigma}\rar&\hm^{hQ}_{1+y\sigma}\rar&\hm^{t Q}_{1+y\sigma}\rar& \ho{\hm}{y\sigma}\rar&\ldots.
\end{tikzcd}
\]
Since $\ho{\hm}{1+y\sigma}=\ho{\hm}{y\sigma}=0$ by Proposition \ref{prop coefficients of hfp and ho}, the top and bottom horizontal arrows in the diagram \eqref{*} are isomorphisms. Thus $u$ acts on $\hm^Q_{1+y\sigma}$ in the same way as $\zeta$, so we need to identify this map.

Note that by $a$-periodicity of $\hm^\Phi$ and $\hm^t$ (see Lemma \ref{lemma_a periodicity of geometric fp and tate}) it is enough to identify this map for $y=0$. Recall from the proof of Lemma \ref{lemma gfp iso to tate} that the fibre $F$ of the map $\epsilon_t\colon\hm^\Phi\to\hm^t$ is $0$-coconnective, in particular $F^Q_1=0$. Thus by Proposition \ref{prop coefficients of Tate} and Theorem \ref{thm coefficients of geometric fixed points} we have that $\zeta$ is an inclusion
\[
\ker(\trm)_Q\hookrightarrow\prescript{}{N}V/(1-\gamma)V.
\]
%Pt 7
\item The proof of this point needs the following diagram:
\[
\begin{tikzcd}
\hm^Q_{\sigma-1}\dar["\epsilon_{\sigma-1}"]\rar["u"] &\hm^Q_{1-\sigma}\dar["\epsilon_{1-\sigma}"] \\
\hm^{hQ}_{\sigma-1}\rar["u","\cong"'] & \hm^{hQ}_{1-\sigma}. 
\end{tikzcd}
\]
By Proposition \ref{prop f_V epsilon_V} the left vertical map $\epsilon_{\sigma-1}$ is the map induced by the restriction of the Mackey functor $\hm^\bullet_{\sigma-1}$. From Theorem \ref{Theorem RO(Q)graded abelian structure of hmq} and Lemma \ref{lemma hm^bullet_sigma-1} we get that this is the map $\coker(\resm)\to (\tilde{V})^Q=\prescript{}{N}V$. This map can be identified with the multiplication by $1+\sigma$ in $\tilde{V}$, thus the multiplication by $1-\gamma$ in $V$. 

Similarly, by Proposition \ref{prop f_V epsilon_V} the right vertical map is induced by the restriction of the Mackey functor $\hm^\bullet_{1-\sigma}$, so by Lemma \ref{lemma hm^bullet_1-sigma} it is an inclusion $\ker(\resm)\to\prescript{}{N}V$. Since by Lemma \ref{lemma ho hfp are uperiodic} the bottom arrow is an isomorphism, the top arrow is a multiplication by $1-\gamma$.
%Pt8
\item By Lemma \ref{lemma general shape} the codomain of $u\colon\hm^Q_{y\sigma-1}\to\hm^Q_{(y-2)\sigma+1}$ is zero when $y>1$.
%Pt9
\item We consider the following commutative diagram:
\[
\begin{tikzcd}
\left(\hm_{hQ}\right)_{y\sigma-3} \rar["u"]\dar & \left(\hm_{hQ}\right)_{(y-2)\sigma-1}\dar \\
\hm^Q_{y\sigma-3}\rar["u"] & \hm^Q_{(y-2)\sigma-1}.
\end{tikzcd}
\]
We use an analogous reasoning as in Point 1 to show that the right arrow is the same as the bottom arrow. 

If $y=3$, then by Proposition \ref{prop f_V epsilon_V} and Lemma \ref{lemma hm^bullet_sigma-1} the vertical left map is induced by the transfer in $\hm^\bullet_{\sigma-1}$. Therefore it is a projection $V/NV\to\coker(\resm)$.
%Pt9
\item For this point we consider the same diagram as in the previous point together with the observation that the bottom arrow is the same as the right vertical arrow. Thus our goal is to describe the right vertical arrow $f_{(y-2)\sigma-1}$ if $y>3$. Note that by Lemma \ref{lemma multiplication by a} we need only to prove the statement for $y=4$. 

By Lemma \ref{lemma multiplication by a} and Observation \ref{obs multiplication by a on ho and hfp} the following diagram is commutative with vertical maps being inclusions (compare with the proof of Corollary \ref{cor fysigma for y>1}):
\[
\begin{tikzcd}
\left(\hm_{hQ}\right)_{2\sigma-1}\rar["f_{2\sigma-1}"]\dar["a"]&\hm^Q_{2\sigma-1}\dar["a"]\\
\left(\hm_{hQ}\right)_{\sigma-1}\rar["f_{\sigma-1}"]&\hm^Q_{\sigma-1}.
\end{tikzcd}
\]
Therefore the top map is the restriction of the projection $V/NV\to\coker(\resm)$ to the domain $V^Q/NV$ and codomain $V^Q/\im(\resm)$.
%Pt11
\item Follows from Lemma \ref{lemma general shape}. In the case $x<0$ and $y<-x$ the multiplication by $u$ starts in the left half-plane below the antidiagonal where all entries are zero, therefore it is the zero map. If $x\geq0$ and $y>-x$, then the multiplication by $u$ lands in the right half-plane above the antidiagonal, where all entries are zero.
%Pt12
\item If $x\leq -2$ or $x\geq 2$ then by Theorem \ref{Theorem RO(Q)graded abelian structure of hmq} we have that $\hm^Q_{x+y\sigma}$ is isomorphic to respectively $\ho{\hm}{x+y\sigma}$ and $\hm^{hQ}_{x+y\sigma}$. Since by Lemma \ref{lemma ho hfp are uperiodic} both $\ho{\hm}{\star}$ and $\hm^{hQ}_\star$ are $u$-periodic, the statement follows.
\end{enumerate}
\end{proof}
\section{Commutativity}
\label{sec comm}
In Section \ref{sec Examples} we have seen that both examples share one feature - the coefficients of both spectra are strictly commutative rings, i.e., the sign coming from the swap of factors is always trivial. In this section we show that this happens for all Green functors. Throughout this section, let $\M$ be a Green functor. Recall that the class of $Q/e$ in $\aq$ is denoted by $\omega$.

\begin{observation}
\label{obs RO(Q)-graded ring}
If $\M$ is a Green functor, then $\hm^Q_\star$ is an $RO(Q)$-graded ring satisfying the following graded commutativity rule: if $\alpha\in\hm^Q_{x+y\sigma}$ and $\beta\in\hm^Q_{x'+y'\sigma}$ then
\[
\alpha\beta=(-1)^{xx'}(1-\omega)^{yy'}\beta\alpha.
\]
Therefore $\hm^Q_\star$ is an \emph{$RO(Q)$-graded commutative ring}.
\end{observation}
\begin{rem}
\label{rem graded comm rule}
Note that we already used the graded commutativity rule in the proofs of Theorems \ref{thm Coefficients of Hz} and \ref{thm coeffs of ha}. For details on the graded commutativity rule in equivariant homotopy theory see \cite[Section 6]{MR764596} or \cite[Lemma 2.12]{MR1808224} specifically for the case of $Q$.
\end{rem}
The Burnside ring $\aq$ acts on $\hm^Q_\star$ as a $0$-th $Q$-homotopy group of the sphere spectrum. By the associativity of a smash product this is given by the action on $\hm^Q_0$:
\[
\pi^Q_0(S^0)\otimes(\hm^Q_0\otimes\hm^Q_V)=(\pi^Q_0(S^0)\otimes\hm^Q_0)\otimes \hm^Q_V\to\hm^Q_V.
\]
By the definition of a box product of Mackey functors we have that $\omega$ acts on $\M(Q/Q)$ as $\trm(1)$:
\[
\omega\cdot 1=\tr_{\underline{A}}(1)\cdot 1=\tr_{\underline{A}\Box\M}(1\otimes(\res_{\M}(1)))=\tr_{\M}(1).
\]
We used here relations in the box product of Mackey functors. Details may be found in \cite[Section 1]{MR979507}. The last equality is obtained by the fact that the restriction in a Green functor is a ring homomorphism, so in particular it preserves the identity.

Recall that $V$ denotes $\M(Q/e)$. Note that if $\M$ is a Green functor then $V$ is a $\M(Q/Q)$-algebra with action given by $\resm$. From this we can deduce that if $v\in V$ then 
\[\omega v=\resm(\trm(1))v=2v,\]
i.e., $\omega$ acts on $V$ as multiplication by $2$. The last equality here comes from the fact that if $\M$ is a Green functor, then $\gamma$ needs to act on $V$ as a unitary ring homomorphism - thus $\resm(\trm(1))=(1+\gamma)1=2$.

\begin{theorem}
\label{thm commutativity}
If $\M$ is a Green functor, then $\hm^Q_\star$ is a strictly commutative ring.
\end{theorem}
Before giving a proof of Theorem \ref{thm commutativity} we prove a couple of preparatory lemmas. 

\begin{lemma}
\label{lemma 2-torsion}
For any Mackey functor $\M$ the groups $\hm^Q_{x+y\sigma}$ are $2$-torsion if $x\neq 0$ and $x\neq-y$.
\end{lemma}
\begin{proof}
Note that if the conditions of Lemma \ref{lemma general shape} hold (i.e., $x+y\sigma$ lies below the antidiagonal on the half-plane $x<0$ or above the antidiagonal on the half-plane $x>0$) then $\hm^Q_{x+y\sigma}$ is zero, so the statement is trivially satisfied.

We will consider four cases depending on the value of $x$.
\begin{enumerate}
\item $x\geq 2$ and $y<-x$. Assume firstly that $y$ is even. By Theorem \ref{Theorem RO(Q)graded abelian structure of hmq} and Proposition \ref{prop coefficients of hfp and ho} we have that
\[
\hm^Q_{x+y\sigma}\cong H^{-x-y}(Q;V).
\]
Note that the group cohomology $H^p(Q;V)$ is $2$-torsion for $p>0$ (this can be easily deduced from \cite[Theorem 6.2.2]{MR1269324}). By the assumption $-x-y\geq 1$, so the statement holds. If $y$ is odd, we have that
\[
\hm^Q_{x+y\sigma}\cong H^{-x-y}(Q;\tilde{V}).
\]
Here we use the fact that $H^p(Q;\tilde{V})\cong H^{p+1}(Q;V)$ (see Section \ref{section hfp, ho and tate}). Thus the claim is proven in this case.
\item $x\leq -2$ and $y>-x$. In this case we proceed analogically to the previous point, using the fact that $\hm^Q_{x+y\sigma}$ is isomorphic to the group homology.
\item $x=1$ and $y<-x$. Then $\hm^Q_{1+y\sigma}\cong\ker(\trm)_Q$. Let $\alpha\in\ker(\trm)_Q$. Then $(1+\gamma)\alpha=\resm(\trm(\alpha))=0$, so $\alpha=-\gamma\alpha$. Thus $2\alpha=\alpha-\gamma\alpha=(1-\gamma)\alpha=0$.
\item $x=-1$ and $y>-x$. Then $\hm^Q_{y\sigma-1}\cong V^Q/\im(\resm)$. Let $\alpha\in V^Q/\im(\resm)$. Then $2\alpha=(1+\gamma)\alpha=\resm(\trm(\alpha))=0$.
\end{enumerate}
\end{proof}
\begin{lemma}
\label{lemma antidiagonal commutative}
The subring $\hm^Q_{\ast-\ast\sigma}$ is strictly commutative.
\end{lemma}
\begin{proof}
Note that unless $x=0$ the groups $\hm^Q_{x-x\sigma}$ are submodules or subquotients of $V$, thus $\omega$ acts on them as $2$. 
Let $\alpha\in\hm^Q_{x-x\sigma}$ and $\beta\in\hm^Q_{x'-x'\sigma}$. By the sign rule the only possibility when a non-trivial sign might occur is when both $x$ and $x'$ are odd. In this case we have
\[
\alpha\beta=(-1)^{xx'}(1-2)^{xx'}\beta\alpha=\beta\alpha.
\]
\end{proof}
\begin{lemma}
\label{lemma hmastsigma commutative}
The subring $\hm^Q_{\ast\sigma}$ is strictly commutative.
\end{lemma}
\begin{proof}
Since the fixed degree is zero we need only to check the possible sign coming from the multiplication by $1-\omega$. We are going to prove that $\omega$ acts as $0$ unless $y\neq 0$. Consider two cases:
\begin{itemize}
\item $y>0$. Then $\hm^Q_{y\sigma}\cong\ker(\resm)$ by Theorem \ref{Theorem RO(Q)graded abelian structure of hmq}. Let $m \in\ker(\resm)$. We have that
\[
\omega\cdot m=\tr(1)\cdot m=\tr(1\cdot\res(m))=0.
\]
So $\omega$ acts as $0$.
\item $y<0$. Then $\hm^Q_{y\sigma}\cong\coker(\trm)$. Thus $\trm(1)=0$ and $\omega$ acts as $0$.
\end{itemize}
Now let $m\in\hm^Q_{y\sigma}$ and $n\in\hm^Q_{y'\sigma}$. If any of $y,y'$ is zero or even then the statement trivially holds. Thus let both $y$ and $y'$ be odd. Then
\[
\alpha\beta=(1-\omega)^{yy'}\beta\alpha=((1-\omega)\beta)\alpha=\beta\alpha.
\]
\end{proof}
\begin{proof}[Proof of Theorem \ref{thm commutativity}]
Let $\alpha\in\hm^Q_{x+y\sigma}$ and $\beta\in\hm^Q_{x'+y'\sigma}$. Then $\alpha\beta\in\hm^Q_{(x+x')+(y+y')\sigma}$. Consider the following cases depending on the degree of the product $\alpha\beta$:
\begin{enumerate}
\item $x+x'\neq 0$ and $x+x'\neq-(y+y')$. Then the product $\alpha\beta$ does not lie on either the antidiagonal or the axis $x=0$. So by Lemma \ref{lemma 2-torsion} the product $\alpha\beta$ belongs to a $2$-torsion group, thus $\alpha\beta=\beta\alpha$.
\item $x=x'=0$. Then we are in the situation of Lemma \ref{lemma hmastsigma commutative}, so the statement holds.
\item $x=-y$ and $x'=-y'$. In this case we use Lemma \ref{lemma antidiagonal commutative}.
\item $x,x'\neq 0$ and $x+x'=0$. If both $\alpha$ and $\beta$ lie on the antidiagonal then the claim is proven by the previous point. Without loss of generality assume that $x\neq -y$, i.e., $\alpha$ does not lie on the antidiagonal. 
By the proof of Lemma \ref{lemma hmastsigma commutative} the possible sign coming from $(1-\omega)^{yy'}$ is trivial since $\omega$ acts trivially on groups lying on $x=0$ axis.
Note that $\alpha$ belongs to a $2$-torsion group, so $\alpha\beta$ is also $2$-torsion and the statement follows.
\item $x+x'=-(y+y')$, i.e., $\alpha\beta$ lies on the antidiagonal. Then we use the same argument as in the previous point.
\end{enumerate} 
\end{proof}
\begin{observation}
\label{obs ha module structure}
By Theorem \ref{thm coeffs of ha} we know that in order to describe the $\ha^Q_\star$-module structure of $\hm^Q_\star$ for any Mackey functor $\M$ we need to describe an action of three elements - $a$, $u$ and $\omega$. Thus we have already described this structure:
\begin{enumerate}
\item the action of $a$ is described by Lemma \ref{lemma multiplication by a};
\item the action of $u$ is described by Theorem \ref{thm multiplication by u};
\item the action of $\omega$ is described in this Section.
\end{enumerate}
\end{observation}
\begin{observation}
\label{obs multiplicative structure}
Let $\M$ be a Green functor. A big part of the multiplicative structure may be derived from the $\ha^Q_\star$-module structure and commutativity described in this Section - we know that there are elements $a_{\hm}=a\cdot 1\in \hm^Q_{-\sigma}$ and $u_{\hm}\in \hm^Q_{2-2\sigma}$ and their multiplicative relations are described by $\ha^Q_\star$-module structure.

This gives a full description of the multiplication in the even fixed degrees. For the odd fixed degrees we proceed as follows:
\begin{enumerate}
\item We need to consider only the elements lying on the antidiagonal, as the relations for other elements follows from Lemma \ref{lemma multiplication by a}.
\item Since the multiplication map $u\colon\hm^Q_{x+y\sigma}\to\hm^Q_{(x+2)+(y-2)\sigma}$ is an isomorphism if $x\geq 3$ or $x\leq -4$, we can restrict our attention to the elements of degrees $1-\sigma$, $3-3\sigma$, $\sigma-1$ and $3\sigma-3$.
\item Relations between these elements can be inferred from the ring maps $\epsilon_\star\colon\hm^Q_\star\to\hm^{hQ}_\star$ and $g_\star\colon\hm^Q_\star\to\hm^{\Phi Q}_\star$.
\end{enumerate}
Examples of such computations are in Section \ref{sec Further examples}.
\end{observation}
\section{Further examples}
\label{sec Further examples}
\subsection{Constant Mackey functor $\ftwo$}
We start this section with the constant Mackey functor $\ftwo$. It has the following structure:
\[
\mackey{\mathbb{F}_2}{\mathbb{F}_2.}{1}{0}
\]
Computations of the coefficients of $\hftwo$ which are built on the unpublished work of Stong appear in \cite{MR1684248} and \cite{MR2699528}, also in work of Hu-Kriz in \cite[Proposition 6.2]{MR1808224}.

\begin{rem}
\label{rem strange notation}
To express the $RO(Q)$-graded ring structure of $\left(\hftwo\right)^Q_\star$ we will use here the notation from \cite[Remark 2.1]{MR4041284}. Therefore we define the following $\mathbb{F}_2[\lambda]$-module:
\[
\frac{\mathbb{F}_2[\lambda]}{\lambda^\infty}:=\colim_k\frac{\mathbb{F}_2[\lambda]}{\lambda^k}.
\]
This is a $\mathbb{F}_2[\lambda]$-module consisting entirely of elements that are infinitely divisible by $\lambda$. By $\frac{\mathbb{F}_2[\lambda]}{\lambda^\infty}\{\theta\}$ we denote the $\mathbb{F}_2[\lambda]$-module consisting of elements of the form $\frac{\theta}{\lambda^k}$ for $k\geq 1$. Note that in particular $\theta\notin\frac{\mathbb{F}_2[\lambda]}{\lambda^\infty}\{\theta\}$. 
\end{rem}
\begin{theorem}
\label{thm hftwo coeffs}
The $RO(Q)$-graded abelian group structure and the multiplicative structure of $\hftwo$ are given by:
\[
\left(\hftwo\right)^Q_\star\cong\mathbb{F}_2[a,\lambda]\oplus\bigoplus_{s\geq 0}\frac{\mathbb{F}_2[\lambda]}{\lambda^\infty}\{\frac{\theta}{a^s}\}.
\]
with $|\theta|=\sigma-1$, $|\lambda|=1-\sigma$ and $|a|=-\sigma$.

This data is presented in Figure \ref{figure hftwo}.
\end{theorem}
\begin{rem}
Note that in the theorem above $\theta$ is a "virtual" element which does not describe any existing element of $\left(\hftwo\right)^Q_\star$. Also note that the product of any two elements from the second summand is zero. The theorem in this form appeared in \cite[Section 2.1]{MR4041284}.
\end{rem}
\begin{figure}[ht]
\begin{tikzpicture}[scale=0.7]
%Kratka
\draw[help lines, thin, lgray] (0,0) grid (20,20);
\draw[->] (0,10)--(20,10) ;
\draw[->] (10,0)--(10,20) ;

%Czerwone linie po lewej
\foreach \x in {0,...,8}
\draw[red] (\x,20)--(\x,20-\x);

%CZerwone linie po prawej
\foreach \x in {0,...,10}
\draw[red] (20-\x,0)--(20-\x,\x);

%Niebieskie linie
\draw[green] (10,10)--(20,0);
\foreach \x in {0,...,8}
\draw[green] (8,12+\x)--(\x,20);
\foreach \x in {1,...,9}
\draw[green] (10,\x)--(10+\x,0);

%Pustekropki w górę
\foreach \x in {0,...,1}
\node[circle, draw, fill=white, scale=0.3] at (1,20-\x) {};
\foreach \x in {0,...,3}
\node[circle, draw, fill=white, scale=0.3] at (3,20-\x) {};
\foreach \x in {0,...,5}
\node[circle, draw, fill=white, scale=0.3] at (5,20-\x) {};
\foreach \x in {0,...,7}
\node[circle, draw, fill=white, scale=0.3] at (7,20-\x) {};
\foreach \x in {0,...,2}
\node[circle, draw, fill=white, scale=0.3] at (2,20-\x) {};
\foreach \x in {0,...,4}
\node[circle, draw, fill=white, scale=0.3] at (4,20-\x) {};
\foreach \x in {0,...,6}
\node[circle, draw, fill=white, scale=0.3] at (6,20-\x) {};
\foreach \x in {0,...,8}
\node[circle, draw, fill=white, scale=0.3] at (8,20-\x) {};

%Pustekropki w dół
\foreach \x in {0,...,2}
\node[circle, draw, fill=white, scale=0.3] at (18,\x) {};
\foreach \x in {0,...,4}
\node[circle, draw, fill=white, scale=0.3] at (16,\x) {};
\foreach \x in {0,...,6}
\node[circle, draw, fill=white, scale=0.3] at (14,\x) {};
\foreach \x in {0,...,8}
\node[circle, draw, fill=white, scale=0.3] at (12,\x) {};
\foreach \x in {0,...,10}
\node[circle, draw, fill=white, scale=0.3] at (10,\x) {};
\foreach \x in {0,...,9}
\node[circle, draw, fill=white, scale=0.3] at (11,\x) {};
\foreach \x in {0,...,7}
\node[circle, draw, fill=white, scale=0.3] at (13,\x) {};
\foreach \x in {0,...,5}
\node[circle, draw, fill=white, scale=0.3] at (15,\x) {};
\foreach \x in {0,...,3}
\node[circle, draw, fill=white, scale=0.3] at (17,\x) {};
\foreach \x in {0,...,1}
\node[circle, draw, fill=white, scale=0.3] at (19,\x) {};
\node[circle, draw, fill=white, scale=0.3] at (20,0) {};
\node[circle, draw, fill=white, scale=0.3] at (0,20) {};

%Etykiety
\foreach \x in {1,...,9}
\node[left] at (10,10-\x) {$a^\x$};
\foreach \x in {1,...,9}
\node[above right] at (10+\x,10-\x) {$\lambda^\x$};

\node[above left] at (20,10) {$x\cdot 1$};
\node[below left] at (10,20) {$y\cdot \sigma$};

\node[scale=1.5] at (10,21) {$\left(\hftwo\right)^Q_{x+y\sigma}$};

\node at (0.5,-1) {Key:};
\node[circle, draw, fill=white, scale=0.3] at (3,-1) {};
\node at (4,-1) {$\mathbb{Z}/2$};
\end{tikzpicture}
\caption{Coefficients of $\hftwo$. The green lines represent multiplication by $\lambda$. Red lines, as before, represent multiplication by $a$.}
\label{figure hftwo}
\end{figure}
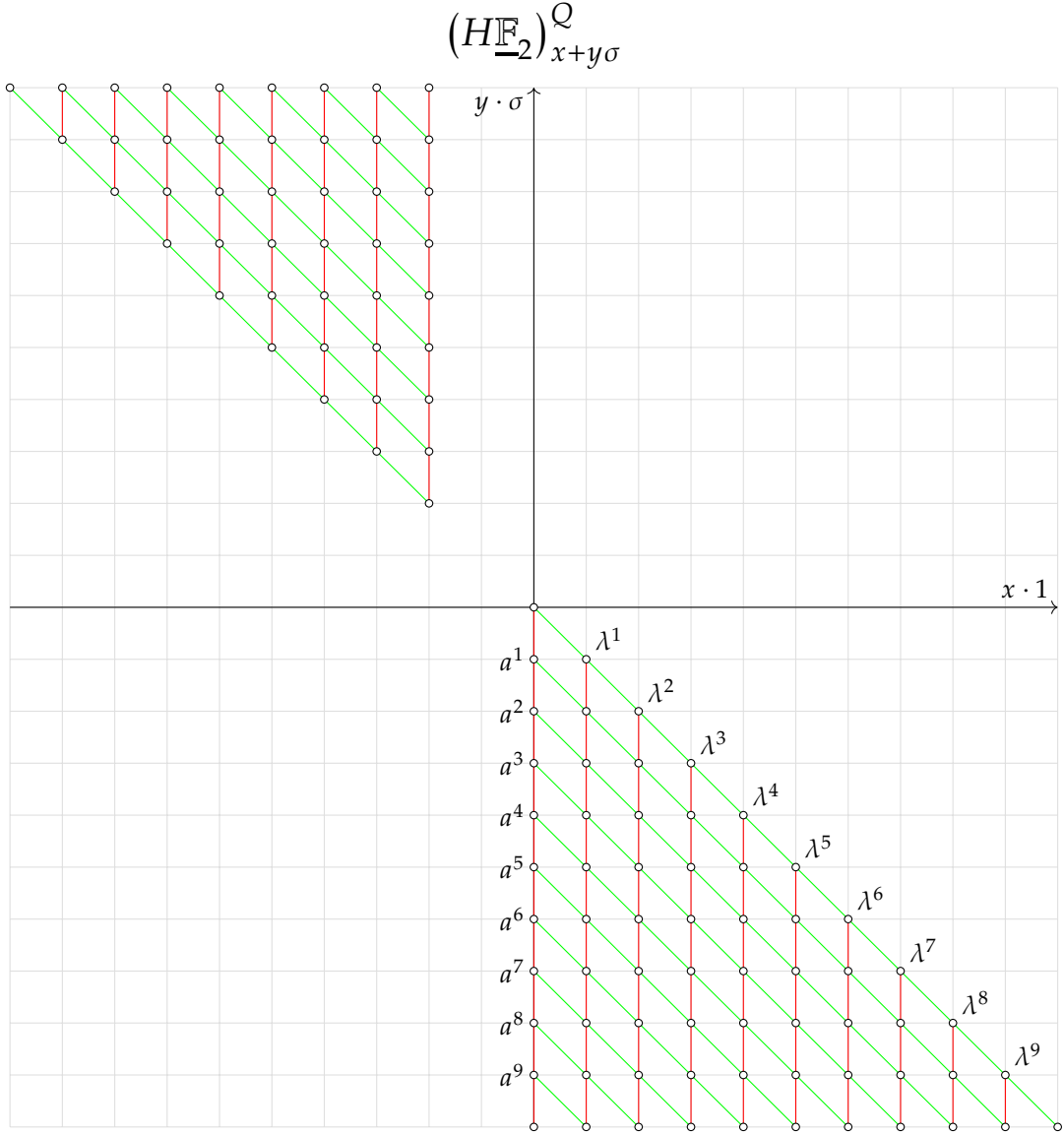
In order to prove this theorem we will need to describe the multiplicative structure of $\left(\hftwo\right)^{hQ}_\star$ and coefficients of $\left(\hftwo\right)_h$.
\begin{lemma}
\label{lemma ho hftwo}
\[
\left(\hftwo\right)^{hQ}_\star\cong\mathbb{F}_2[a][\lambda,\lambda^{-1}].
\]
\end{lemma}
\begin{proof}
The $RO(Q)$-graded abelian group structure follows from Proposition \ref{prop coefficients of hfp and ho}. To see the multiplicative structure we are going to describe the $E_1$-page of the trigraded homotopy fixed point spectral sequence, following the discussion given in Subsection \ref{subsec hfp ho spectral sequence}.

Fix $y$. Recall from Subsection \ref{subsec hfp ho spectral sequence} that for $X=\hftwo$:
\[
E_1^{pq}(y)=\Hom_Q\left(H_p(Q_+\wedge S^p),\pi_{-q}\left(F\left(S^{y\sigma},\hftwo\right)\right)\right)
\]
where $\pi_{-q}\left(F\left(S^{y\sigma},\hftwo\right)\right)$ denotes the $-q$-th homotopy group of the underlying naive spectrum of $F\left(S^{y\sigma},\hftwo\right)$.

We have that $H_p(Q_+\wedge S^p)\cong\mathbb{Z}[Q]$ as a $\mathbb{Z}[Q]$-module and
\[
\pi_{-q}\left(F\left(S^{y\sigma},\hftwo\right)\right)\cong\left\{
\begin{array}{ll}
\mathbb{Z}/2&\textrm{if}\; y=q \\
0 &\textrm{else.}
\end{array}
\right.
\]
The differentials on the $E_1$-page are the differentials computing the group cohomology, thus they are all $0$. So the spectral sequence collapses on the $E_1$-page.

The trigraded homotopy fixed point spectral sequence is multiplicative, so we can describe the $E_1$-page as an algebra as follows:
\[
E_1^{**}(*)\cong\mathbb{F}_2[a][\lambda,\lambda^{-1}]
\]
with $|a|=(1,-1,-1)$ and $|\lambda|=(0,-1,-1)$. Since all differentials are $0$, we get that
\[
\left(\hftwo\right)^{hQ}_\star\cong \mathbb{F}_2[a][\lambda,\lambda^{-1}]
\]
with $|a|=-\sigma$ and $|\lambda|=1-\sigma$.
\end{proof} 
\begin{cor}
\label{cor hoht hftwo}
\[
\begin{array}{c}
\left(\hftwo\right)^{tQ}_\star\cong\mathbb{F}_2[a,a^{-1}][\lambda,\lambda^{-1}] \\
\ho{\left(\hftwo\right)}{\star}\cong\mathbb{F}_2[a^{-1}][\lambda,\lambda^{-1}].
\end{array}
\]
\end{cor}
\begin{proof}
The first statement follows from the fact that $\hftwo^t\simeq\hftwo^h\wedge\widetilde{EQ}$ and from Lemma \ref{lemma ho hfp are uperiodic}. The second part follows from the long exact sequence in homotopy for the cofibre sequence 
\[
\left(\hftwo\right)_h\to\hftwo^h\to \hftwo^t
\]
where the map $\ho{\left(\hftwo\right)}{\star}\to\left(\hftwo\right)^{hQ}_\star$ is multiplication by $N$, so zero in this case.
\end{proof}
\begin{proof}[Proof of Theorem \ref{thm hftwo coeffs}]
The $RO(Q)$-graded abelian group structure follows from the Theorem \ref{Theorem RO(Q)graded abelian structure of hmq}. Following Observation \ref{obs multiplicative structure} we know that the multiplicative structure of $\left(\hftwo\right)^Q_\star$ is described by the following elements:
\begin{itemize}
\item $a=a_{\hftwo}$ - for the properties of multiplication by $a$ see Lemma \ref{lemma multiplication by a};
\item $u=u_{\hftwo}$ - see Section \ref{sec multiplication by u};
\item $\lambda\in\left(\hftwo\right)^Q_{1-\sigma}$.
\end{itemize}
We need to check three relations that do not follow directly from previous sections - i.e., we need to prove that:
\begin{enumerate} 
\item $\lambda^2=u$.
\item If $\eta$ is the generator of $\ha^Q_{2\sigma-2}$ then $\hm^Q_{3\sigma-3}$ is generated by $\lambda^{-1}\eta$.
\item $\eta^2=0$.
\end{enumerate}
The rest of the structure will follow.
  
We firstly prove point (1), thus we are going to prove that $\lambda^2=u$. Consider the diagram expressing multiplication by $\lambda$:
\[
\begin{tikzcd}
\left(\hftwo\right)^Q_{1-\sigma}\dar["\epsilon_{1-\sigma}"]\rar["\lambda"] &\left(\hftwo\right)^Q_{2-2\sigma}\dar["\epsilon_{2-2\sigma}"] \\
\left(\hftwo\right)^{hQ}_{1-\sigma}\rar["\lambda"] & \left(\hftwo\right)^{hQ}_{2-2\sigma}.
\end{tikzcd}
\]
By Theorem \ref{Theorem RO(Q)graded abelian structure of hmq}, the right vertical arrow is an isomorphism, so does the bottom horizontal arrow by Lemma \ref{lemma ho hftwo}. From Lemma \ref{lemma hm^bullet_1-sigma} we deduce that the left vertical arrow is also an isomorphism. Thus the top vertical arrow is an isomorphism and $\lambda^2=u$. 

Now proceed to the point (2). Let $\eta$ be a generator of $\left(\hftwo\right)^Q_{2\sigma-2}$ and $\eta '$ a generator of $\left(\hftwo\right)^Q_{3\sigma-3}$. Similar argument as above applied to the diagram
\[
\begin{tikzcd}
\ho{\left(\hftwo\right)}{3\sigma-3}\dar["f_{\sigma-1}"]\rar["\lambda"] &\ho{\left(\hftwo\right)}{2\sigma-2}\dar["f_{2\sigma-2}"] \\
\left(\hftwo\right)^Q_{3\sigma-3}\rar["\lambda"] & \left(\hftwo\right)^Q_{2\sigma-2}
\end{tikzcd}
\]
together with Corollary \ref{cor hoht hftwo} shows that $\lambda^{-1}\eta=\eta '$.

From the relations above we derive the point (3), i.e., that $\eta^2=0$. Note that $\lambda\eta=0$. By multiplying the relation from the previous point $\eta=\lambda\eta'$ by $\eta$ and using commutativity, we get that $\eta^2=0$.

The statement in the Theorem is obtained by putting $\eta=\frac{\theta}{\lambda}$.
\end{proof}

\subsection{The norm of $\mathbb{F}_2$, $N_e^Q\mathbb{F}_2$}
We continue the examples section with the Mackey functor $N_e^Q\mathbb{F}_2$. It has the form
\[
\mackey{\mathbb{Z}/4}{\mathbb{Z}/2.}{1}{2}
\]
This Mackey functor appears in the work of Hill in \cite{2017arXiv170902005H}, where he computes the Bredon homology with coefficients in it of spaces of the form $\Omega^\sigma\Sigma^\sigma X$. We will follow the notation from this paper and denote this Mackey functor by $\underline{B}$. For us $\underline{B}$ is an example of an interesting feature - it has only two zero columns, which are in fixed degrees $1$ and $-1$.

Since $\underline{B}(Q/e)=\ftwo(Q/e)$, Lemma \ref{lemma ho hftwo} and Corollary \ref{cor hoht hftwo} describe also $\hb^{hQ}_\star$, $\ho{\hb}{\star}$ and $\hb^{tQ}_\star$. Most of the proofs follow in the analogous way as in the previous examples, so we will only comment on how to get the multiplicative structure of $\hb^Q_\star$. 
\begin{lemma}
\label{lemma gfp of hb}
\[
\hb^{\Phi Q}_\star\cong \mathbb{F}_2[a,a^{-1}]\oplus\lambda^{2}\mathbb{F}_2[a,a^{-1},\lambda]
\]
with $|a|=-\sigma$ and $|\lambda|=1-\sigma$.
\end{lemma}
\begin{theorem}
The $RO(Q)$-graded abelian group structure and the multiplicative structure of $\hb^Q_\star$ is given by:
\[
\hb^Q_\star\cong \frac{\mathbb{Z}/4[a]}{2a}[u,u\lambda]\oplus\frac{\mathbb{Z}_4[a,u,u\lambda]}{(a^\infty,u^\infty,(u\lambda)^\infty)}\{2\}.
\]
Here $|a|=-\sigma$, $|\lambda|=1-\sigma$, $|u|=2-2\sigma$, $u=\lambda^2$ and the second direct summand is a $\mathbb{Z}_4[a,u,u\lambda]$-module consisting of elements of the form 
\[\frac{2}{a^ku^l(u\lambda)^m}\] for $k,l,m\geq 0$ not simultaneously equal to $0$ (see Remark \ref{rem strange notation}).

This data is presented in Figure \ref{figure coeffs of hb}.
\end{theorem}
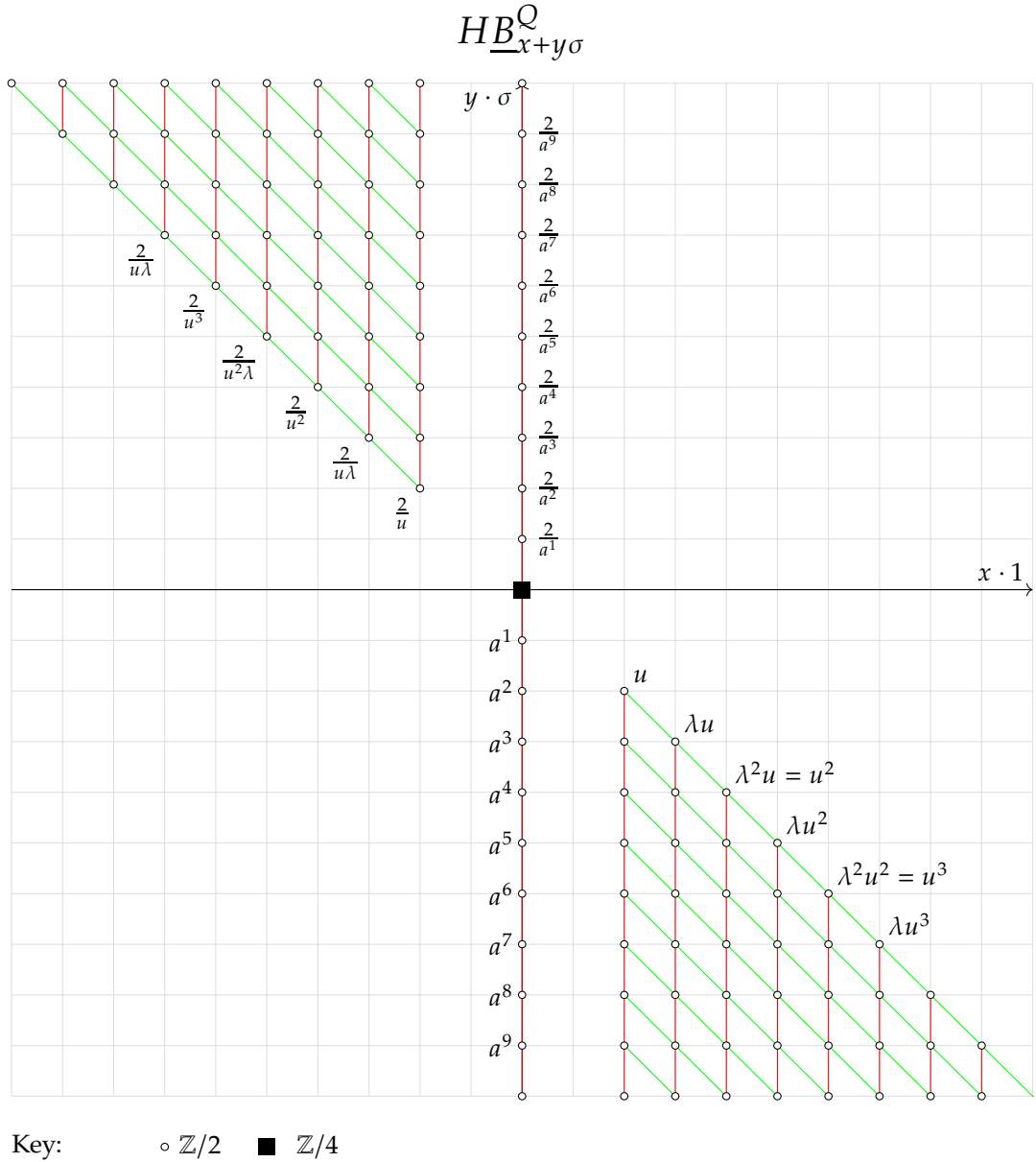
\begin{figure}[ht]
\begin{tikzpicture}[scale=0.7]
%Kratka
\draw[help lines, thin, lgray] (0,0) grid (20,20);
\draw[->] (0,10)--(20,10) ;
\draw[->] (10,0)--(10,20) ;
%Czerwone linie po lewej
\foreach \x in {0,...,8}
\draw[red] (\x,20)--(\x,20-\x);

%CZerwone linie po prawej
\foreach \x in {0,...,8}
\draw[red] (20-\x,0)--(20-\x,\x);
\draw[red] (10,0)--(10,20);
%Niebieskie linie
\foreach \x in {0,...,8}
\draw[green] (8,12+\x)--(\x,20);
\foreach \x in {0,...,8}
\draw[green] (12,8-\x)--(20-\x,0);

%Pustekropki w górę
\foreach \x in {0,...,1}
\node[circle, draw, fill=white, scale=0.3] at (1,20-\x) {};
\foreach \x in {0,...,3}
\node[circle, draw, fill=white, scale=0.3] at (3,20-\x) {};
\foreach \x in {0,...,5}
\node[circle, draw, fill=white, scale=0.3] at (5,20-\x) {};
\foreach \x in {0,...,7}
\node[circle, draw, fill=white, scale=0.3] at (7,20-\x) {};
\foreach \x in {0,...,2}
\node[circle, draw, fill=white, scale=0.3] at (2,20-\x) {};
\foreach \x in {0,...,4}
\node[circle, draw, fill=white, scale=0.3] at (4,20-\x) {};
\foreach \x in {0,...,6}
\node[circle, draw, fill=white, scale=0.3] at (6,20-\x) {};
\foreach \x in {0,...,8}
\node[circle, draw, fill=white, scale=0.3] at (8,20-\x) {};

%Pustekropki w dół
\foreach \x in {0,...,2}
\node[circle, draw, fill=white, scale=0.3] at (18,\x) {};
\foreach \x in {0,...,4}
\node[circle, draw, fill=white, scale=0.3] at (16,\x) {};
\foreach \x in {0,...,6}
\node[circle, draw, fill=white, scale=0.3] at (14,\x) {};
\foreach \x in {0,...,8}
\node[circle, draw, fill=white, scale=0.3] at (12,\x) {};
\foreach \x in {0,...,9}
\node[circle, draw, fill=white, scale=0.3] at (10,\x) {};
\foreach \x in {11,...,20}
\node[circle, draw, fill=white, scale=0.3] at (10,\x) {};
\foreach \x in {0,...,7}
\node[circle, draw, fill=white, scale=0.3] at (13,\x) {};
\foreach \x in {0,...,5}
\node[circle, draw, fill=white, scale=0.3] at (15,\x) {};
\foreach \x in {0,...,3}
\node[circle, draw, fill=white, scale=0.3] at (17,\x) {};
\foreach \x in {0,...,1}
\node[circle, draw, fill=white, scale=0.3] at (19,\x) {};
\draw (20,0) {};
\node[circle, draw, fill=white, scale=0.3] at (0,20) {};

%Etykiety
\foreach \x in {1,...,9}
\node[left] at (10,10-\x) {$a^\x$};
\foreach \x in {1,...,9}
\node[right,xshift=2pt] at (10,10+\x) {$\frac{2}{a^\x}$};
\node[above right] at (12,8) {$u$};
\node[above right] at (13,7) {$\lambda u$};
\node[above right] at (14,6) {$\lambda^2 u=u^2$};
\node[above right] at (15,5) {$\lambda u^2$};
\node[above right] at (16,4) {$\lambda^2 u^2=u^3$};
\node[above right] at (17,3) {$\lambda u^3$};

\node[below left] at (8,12) {$\frac{2}{u}$};
\node[below left] at (7,13) {$\frac{2}{u\lambda}$};
\node[below left] at (6,14) {$\frac{2}{u^2}$};
\node[below left] at (5,15) {$\frac{2}{u^2\lambda}$};
\node[below left] at (4,16) {$\frac{2}{u^3}$};
\node[below left] at (3,17) {$\frac{2}{u\lambda}$};
\node[regular polygon,regular polygon sides=4,fill,scale=0.7] at (10,10) {};

\node[above left] at (20,10) {$x\cdot 1$};
\node[below left] at (10,20) {$y\cdot \sigma$};

\node[scale=1.5] at (10,21) {$H\underline{B}^Q_{x+y\sigma}$};

\node at (0.5,-1) {Key:};
\node[circle, draw, fill=white, scale=0.3] at (3,-1) {};
\node at (3.7,-1) {$\mathbb{Z}/2$};
\node[regular polygon,regular polygon sides=4,fill,scale=0.7] at (5,-1) {};
\node at (6,-1) {$\mathbb{Z}/4$};
\end{tikzpicture}
\caption{Coefficients of $\hb$. Notation as in Figure \ref{figure hftwo}.}
\label{figure coeffs of hb}
\end{figure}
\begin{proof}
We are going to comment only on the multiplicative structure. By Observation \ref{obs multiplicative structure} we know that it can be described by the following elements:
\begin{itemize}
\item $a_{\hb}\in\hb^Q_{-\sigma}$;
\item $u_{\hb}\in\hb^Q_{2-2\sigma}$;
\item a generator of $\hb^Q_{3-3\sigma}$.
\end{itemize}
and their inverses. From the map $\epsilon_\star\colon\hb^Q_\star\to\hb^{hQ}_\star$ we can deduce that the last element may be described as $\lambda^3=u\lambda$. Thus the result follows.
\end{proof}
\subsection{The Mackey functor $\zt$}
We close the examples section with an example of a Mackey functor with a non-trivial action on the $Q/e$-level. This is the fixed points Mackey functor of $\tilde{\mathbb{Z}}$ and its structure is given by:
\[
\mackey{0}{\tilde{\mathbb{Z}}.}{}{}
\]
We denote this Mackey functor by $\zt$. Since $\gamma$ does not act on the $Q/e$-level as a unitary ring homomorphism, it is not a Green functor. Therefore we describe only the $\ha^Q_\star$-module structure of $\hzt^Q_\star$.
\begin{lemma}
\[
\begin{array}{c}
\ho{\hzt}{\star}\cong\ho{\hz}{\star+1-\sigma} \\
\hzt^{hQ}_\star\cong\hz^{hQ}_{\star+1-\sigma}.
\end{array}
\]
\end{lemma}
\begin{proof}
Follows from the fact that $\tilde{\tilde{\mathbb{Z}}}\cong\mathbb{Z}$ as $\mathbb{Z}[Q]$-modules and Proposition \ref{prop coefficients of hfp and ho}.
\end{proof}

The $\ha^Q_\star$-module structure of $\hzt^Q_\star$ is depicted in Figure \ref{figure coeffs of hzt}. Multiplication by $u$ is represented by blue dashed lines, whereas multiplication by $a$ is given by red dashed lines. Note that since the $Q$-action on the $Q/e$-level of $\underline{\tilde{\mathbb{Z}}}$ is non-trivial, entries on $1-\sigma$ and $\sigma-1$ spots are different from the rest of the $x=1$ and $x=-1$ columns, respectively.
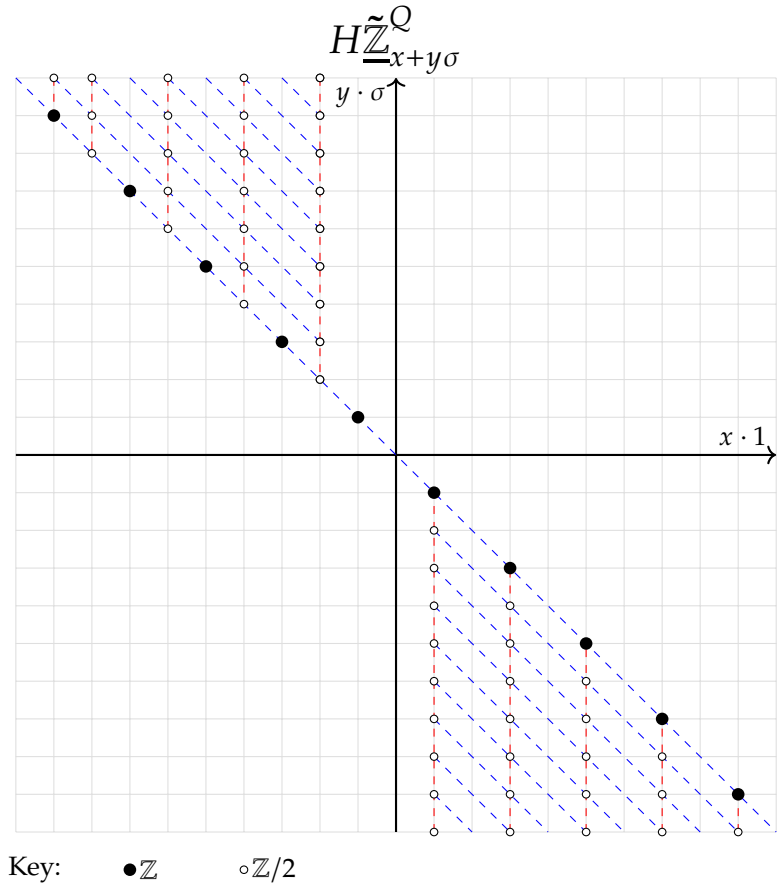
\begin{figure}[ht]
\begin{tikzpicture}[scale=0.5]
%Kratka
\draw[help lines, thin, lgray] (0,0) grid (20,20);
\draw[->, thick] (0,10)--(20,10) ;
\draw[->, thick] (10,0)--(10,20) ;

%Czerwone linie po lewej
\draw[red, dashed] (2,20)--(2,18);
\draw[red, dashed] (4,20)--(4,16);
\draw[red, dashed] (6,20)--(6,14);
\draw[red, dashed] (8,20)--(8,12);

%CZerwone linie po prawej
\draw[red, dashed] (11,0)--(11,9);
\draw[red, dashed] (13,0)--(13,7);
\draw[red, dashed] (15,0)--(15,5);
\draw[red, dashed] (17,0)--(17,3);
\draw[red, dashed] (19,0)--(19,1);
\draw[red, dashed] (1,20)--(1,19);

%Niebieskie linie
\draw[blue, dashed] (0,20)--(20,0);
\foreach \x in {1,...,7}
\draw[blue, dashed] (\x,20)--(8,12+\x);

\foreach \x in {0,...,8}
\draw[blue, dashed] (11,\x)--(11+\x,0);

%Grube kropki
\foreach \x in {0,...,9}
\pgfmathsetmacro\r{\x*2+1}
\node[circle,fill,scale=0.5] at (\r,20-\r) {};

%Pustekropki w górę
\foreach \x in {0,...,2}
\node[circle, draw, fill=white, scale=0.3] at (2,20-\x) {};
\foreach \x in {0,...,4}
\node[circle, draw, fill=white, scale=0.3] at (4,20-\x) {};
\foreach \x in {0,...,6}
\node[circle, draw, fill=white, scale=0.3] at (6,20-\x) {};
\foreach \x in {0,...,8}
\node[circle, draw, fill=white, scale=0.3] at (8,20-\x) {};
\foreach \x in {0,...,8}
\node[circle, draw, fill=white, scale=0.3] at (8,20-\x) {};
\node[circle, draw, fill=white, scale=0.3] at (1,20) {};
\node[circle, draw, fill=white, scale=0.3] at (19,0) {};

%Pustekropki w dół
\foreach \x in {0,...,2}
\node[circle, draw, fill=white, scale=0.3] at (17,\x) {};
\foreach \x in {0,...,4}
\node[circle, draw, fill=white, scale=0.3] at (15,\x) {};
\foreach \x in {0,...,6}
\node[circle, draw, fill=white, scale=0.3] at (13,\x) {};
\foreach \x in {0,...,8}
\node[circle, draw, fill=white, scale=0.3] at (11,\x) {};

\node[above left] at (20,10) {$x\cdot 1$};
\node[below left] at (10,20) {$y\cdot \sigma$};

\node[scale=1.5] at (10,21) {$H\underline{\tilde{\mathbb{Z}}}^Q_{x+y\sigma}$};

\node at (0.5,-1) {Key:};
\node[circle, fill,scale=0.5] at (3,-1) {};
\node at (3.5,-1) {$\mathbb{Z}$};
\node[circle, draw, fill=white, scale=0.3] at (6,-1) {};
\node at (6.8,-1) {$\mathbb{Z}/2$};
\end{tikzpicture}
\caption{Coefficients of $\hzt$.}
\label{figure coeffs of hzt}
\end{figure}

\bibliographystyle{amsplain}

\begin{thebibliography}{10}

\bibitem{MR764596}
J.~F. Adams, \emph{Prerequisites (on equivariant stable homotopy) for
  {C}arlsson's lecture}, Algebraic topology, {A}arhus 1982 ({A}arhus, 1982),
  Lecture Notes in Math., vol. 1051, Springer, Berlin, 1984, pp.~483--532.
  \MR{764596}

\bibitem{MR0214062}
Glen~E. Bredon, \emph{Equivariant cohomology theories}, Lecture Notes in
  Mathematics, No. 34, Springer-Verlag, Berlin-New York, 1967. \MR{0214062}

\bibitem{MR672956}
Kenneth~S. Brown, \emph{Cohomology of groups}, Graduate Texts in Mathematics,
  vol.~87, Springer-Verlag, New York-Berlin, 1982. \MR{672956}

\bibitem{MR1684248}
Jeffrey~L. Caruso, \emph{Operations in equivariant
  {$\mathbb{Z}/p$-cohomology}}, Math. Proc. Cambridge Philos. Soc. \textbf{126}
  (1999), no.~3, 521--541. \MR{1684248}

\bibitem{MR2240234}
Daniel Dugger, \emph{An {A}tiyah-{H}irzebruch spectral sequence for
  {$KR$}-theory}, $K$-Theory \textbf{35} (2005), no.~3-4, 213--256 (2006).
  \MR{2240234}

\bibitem{MR2025457}
Kevin~K. Ferland and L.~Gaunce Lewis, Jr., \emph{The {$R{O}(G)$}-graded
  equivariant ordinary homology of {$G$}-cell complexes with even-dimensional
  cells for {$G=\mathbb{Z}/p$}}, Mem. Amer. Math. Soc. \textbf{167} (2004),
  no.~794, viii+129. \MR{2025457}

\bibitem{MR2699528}
Kevin~Keith Ferland, \emph{On the {RO}({G})-graded equivariant ordinary
  cohomology of generalized {G}-cell complexes for {G} = {Z}/p}, ProQuest LLC,
  Ann Arbor, MI, 1999, Thesis (Ph.D.)--Syracuse University. \MR{2699528}

\bibitem{FourApproaches}
J.~P.~C. Greenlees, \emph{Four approaches to cohomology theories with reality},
  An alpine bouquet of algebraic topology, Contemp. Math., vol. 708, Amer.
  Math. Soc., Providence, RI, 2018, pp.~139--156. \MR{3807754}

\bibitem{MR1230773}
J.~P.~C. Greenlees and J.~P. May, \emph{Generalized {T}ate cohomology}, Mem.
  Amer. Math. Soc. \textbf{113} (1995), no.~543, viii+178. \MR{1230773}

\bibitem{MR3709655}
J.~P.~C. Greenlees and Lennart Meier, \emph{Gorenstein duality for real
  spectra}, Algebr. Geom. Topol. \textbf{17} (2017), no.~6, 3547--3619.
  \MR{3709655}

\bibitem{MR4041284}
Bertrand~J. Guillou, Michael~A. Hill, Daniel~C. Isaksen, and Douglas~Conner
  Ravenel, \emph{The cohomology of {$C_2$}-equivariant {$\mathcal{A}(1)$} and
  the homotopy of {$\mathrm{ko}_{C_2}$}}, Tunis. J. Math. \textbf{2} (2020),
  no.~3, 567--632. \MR{4041284}

\bibitem{MR3505179}
M.~A. Hill, M.~J. Hopkins, and D.~C. Ravenel, \emph{On the nonexistence of
  elements of {K}ervaire invariant one}, Ann. of Math. (2) \textbf{184} (2016),
  no.~1, 1--262. \MR{3505179}

\bibitem{2017arXiv170902005H}
Michael~A. {Hill}, \emph{{On the algebras over equivariant little disks}},
  arXiv e-prints (2017), arXiv:1709.02005.

\bibitem{MR1808224}
Po~Hu and Igor Kriz, \emph{Real-oriented homotopy theory and an analogue of the
  {A}dams-{N}ovikov spectral sequence}, Topology \textbf{40} (2001), no.~2,
  317--399. \MR{1808224}

\bibitem{MR598689}
G.~Lewis, J.~P. May, and J.~McClure, \emph{Ordinary {$RO(G)$}-graded
  cohomology}, Bull. Amer. Math. Soc. (N.S.) \textbf{4} (1981), no.~2,
  208--212. \MR{598689}

\bibitem{MR979507}
L.~Gaunce Lewis, Jr., \emph{The {$R{O}(G)$}-graded equivariant ordinary
  cohomology of complex projective spaces with linear {$ \mathbb{Z}/p$}
  actions}, Algebraic topology and transformation groups ({G}\"{o}ttingen,
  1987), Lecture Notes in Math., vol. 1361, Springer, Berlin, 1988,
  pp.~53--122. \MR{979507}

\bibitem{MR1269324}
Charles~A. Weibel, \emph{An introduction to homological algebra}, Cambridge
  Studies in Advanced Mathematics, vol.~38, Cambridge University Press,
  Cambridge, 1994. \MR{1269324}

\bibitem{zeng2018equivariant}
Mingcong Zeng, \emph{Equivariant {E}ilenberg-{M}ac{L}ane spectra in cyclic
  $p$-groups}, arXiv Mathematics e-prints (2018), arXiv:1710.01769.

\end{thebibliography}
\end{document}